\documentclass{article}
\usepackage{graphicx}

\usepackage[fleqn]{amsmath}

\usepackage{amsthm}
\usepackage{amssymb}
\usepackage{amsbsy}
\usepackage{amsmath}
\usepackage{amsfonts}
\usepackage{mathrsfs}
\usepackage{amsmath}
\usepackage[all]{xy}
\usepackage{amstext}
\usepackage{amscd}
\usepackage[dvips]{epsfig}
\usepackage{psfrag}
\usepackage{enumerate}
\usepackage{flafter}
\allowdisplaybreaks

\theoremstyle{plain}
\newtheorem{thm}{Theorem}[section]
\newtheorem{prop}[thm]{Proposition}

\newtheorem{lem}[thm]{Lemma}
\theoremstyle{definition}
\newtheorem{exa}[thm]{Example}

\newtheorem{defn}[thm]{Definition}
\newtheorem{prob}[thm]{Problem}
\newtheorem{fact}[thm]{Fact}

\def\4{\mathop{4 \mathrm{f}}\nolimits}

\def\12{\mathop{X_{12}}\nolimits}

\def\dim{\mathop{\mathrm{dim}}\nolimits}

\def\det{\mathop{\mathrm{det}}\nolimits}
\def\Im{\mathop{\mathrm{Im}}\nolimits}
\def\Ker{\mathop{\mathrm{Ker}}\nolimits}

\def\Hom{\mathop{\mathrm{Hom}}\nolimits}

\newcommand{\mb}[1]{{\mathbf{#1}}}

\newcommand{\GL}{{\rm GL}}
\newcommand{\SL}{{\rm SL}}
\newcommand{\SU}{{\rm SU}}

\newcommand{\lra}{\longrightarrow}
\newcommand{\ra}{\rightarrow}
\newcommand{\Q}{{\Bbb Q}}
\newcommand{\R}{{\Bbb R}}
\newcommand{\Z}{{\Bbb Z}}
\newcommand{\N}{{\Bbb N}}
\newcommand{\F}{{\Bbb F}}
\newcommand{\C}{{\Bbb C}}

\usepackage{color}

\begin{document}
\large
\begin{center}
{\bf\Large Non-acyclic Reidemeister torsions of manifolds of odd dimension}
\end{center}
\begin{center}{
Takefumi Nosaka
\footnote{E-mail address: {\tt nosaka.t.aa@m.titech.ac.jp}},
Naoko Wakijo
\footnote{E-mail address: {\tt wakijo.n.aa@m.titech.ac.jp}},
and Koki Yanagida
\footnote{E-mail address: {\tt yanagida.k.ab@m.titech.ac.jp}}
}\end{center}
\begin{abstract}\baselineskip=12pt \noindent
Given an oriented closed $C^{\infty}$-manifold $M$ of odd dimension and a unitary representation $\rho : \pi_1(M) \ra \GL_n(\F)$, we define a Reidemeister torsion, even if the cohomology associated with $\rho$ is not acyclic. As corollaries, we introduce some topological invariants of $M$, which include non-acyclic extensions of abelian torsions and the Alexander polynomials of links. Furthermore, we propose a volume form of the $\SU(n)$-character varieties of $M$.
Moreover, we compute the Reidemeister torsions of some representations of 3-manifolds and give a comparison of the works of Farber and Turaev.
\end{abstract}
\begin{center}
\normalsize
\baselineskip=11pt
{\bf Keywords} \\
Reidemeister torsion, 3-manifolds, Poincar\'{e} duality \ \ \
\end{center}
\begin{center}
\normalsize
\baselineskip=11pt
{\bf Subject Codes } \\
\ \ \ \ \ \ 19J10, 58K65, 57P10, 15A63 \ \ \
\end{center}

\section{Introduction}
Let $\F$ be a field with involution $\bar{}: \F \ra \F$, and let $C^*$ be a bounded chain complex consisting of finite dimensional $\F$-vector spaces.
Roughly speaking, a combinatorial (Reidemeister) torsion of $C^*$ is an alternating product of determinants of $\delta^i$'s (see \S \ref{reidemeister1} for the definition).
For a closed oriented $C^{\infty}$-manifold $X$,
if the cellular complex $C^*_\rho(X)$ of $X$ with local coefficients over $\F$ is acyclic and
satisfies certain conditions, the torsion of $C^*_\rho(X)$ is a topological invariant and studied from many aspects;
see, e.g., \cite{Tur,Nic} for its applications and developments.
Meanwhile, if the complex $C_*^\rho(X)$ is not acyclic, with a choice of the basis of cohomology $H^*_\rho(X) $,
we can define a (non-acyclic) torsion in $ \F $ that depends on the choice.
By the parity of Poincar\'{e} duality, the existing approaches of the non-acyclic torsion
are different according to the parity of $\dim (X)$; see \cite{Sozen, Wit} and the references therein for even cases of $\dim (X) $,
and see \cite{PY} for knot cases.
Let us focus on non-acyclic torsions with $ \dim X= 2m+1$ for some $m \in \N $.
Inspired by analytic approaches to the torsions as in the Ray-Singer norm,
Farber and Turaev \cite{Fa, FT} define a torsion on the determinant line of the cohomology of any flat vector bundle over $X$, and a {\it Poincar\'{e}-Reidemeister scalar product} from their torsion; see Appendix \ref{999166}.
In the definition, they assume $\F= \R$ or $\F=\C$, and also use the duality from the twisted intersection forms on homology.
However, the definition of the twisted intersection form had ambiguous aspects, and the above determinant line is not quantitatively investigated; furthermore, few non-trivial examples of computing the scalar products are known; see, e.g., \cite{Fa,FT,NS}.

In this paper, from a cohomological viewpoint with a Poincar\'{e} duality,
we will define combinatorially a non-acyclic torsion with respect to a
representation $\rho: \pi_1(X) \ra \GL_n(\F)$ with $ \dim X= 2m+1$, where $\rho$ is required to admit a bilinear form satisfying a certain condition; see Theorem \ref{l33l24}. 
The point of the definition is that $\F$ may be generally a field of characteristic zero, and the torsion is independent of the choice of the basis of
$H^*_\rho(X) $; see Theorem \ref{l33l24}. 
Furthermore, we show (Appendix \ref{999166}) that the torsion is a (square-root) modification of the Poincar\'{e}-Reidemeister scalar product.
However, because the non-acyclic torsion is defined in a quotient group of the form $\F^{\times }/\{ x \bar{x} \mathrm{det}(\rho (y))\mid x \in \F^{\times }, y \in \pi_1(X)\}$, the quotient group may be trivial in some cases; see Section \ref{HddG3} for some computations of the group.
To obtain non-trivial torsions, we shall restrict ourselves to cases in which the quotient group is not trivial.

In such non-trivial cases,
we suggest applications from the non-acyclic torsions to some topological invariants of a closed $C^{\infty}$-manifold $X$ of odd dimension; see Section \ref{HG33}.
When $\rho$ is trivial, our torsion can be expressed in terms of the ordinary cohomology; see Sections \ref{HG444j44} and \ref{HG42442}.
For non-trivial cases of $\rho$, we see that the torsion above produces abelian Reidemeister torsions in non-acyclic cases (Section \ref{HG4464}).
As an application, we propose a generalization of the Alexander polynomial of links in Section \ref{HG43344}.
As another application, we ensure that the torsion generates a volume form on the $\SU(n)$-character variety of $X$, as an analogy to the volume form of the flat moduli spaces of closed surfaces \cite{Wit} (see Section \ref{HG444}).

Finally, in Section \ref{HG2}, we focus on the case of $\dim X=3$ and provide examples of computing the torsions of 3-manifolds.
In the 3-dimensional case,
we can compute the Poincar\'{e} duality because some procedures for computing the cup product are known \cite{SieradskiIM86, TrotterAM62}; in doing so, we obtain non-trivial torsions in some cases.
In addition, we argue further refinements of the torsions; see Propositions \ref{l32l245} and \ref{l32l2451}.
Although this paper suggests quantitative definitions and research of non-acyclic torsions, their applications and potential properties need to be developed in future studies.





\

\noindent
{\bf Conventional terminology.}
Throughout this paper, $\F$ and $\F^{\times }= \F\setminus\{0\}$ denote a commutative field and the multiplicative group, respectively. Furthermore, we denote by $\bar{ }:\F \ra \F$ an involution on $\F$ (possibly, $\ \bar{}=\mathrm{id}_{\F }$).
By $X$, we mean an oriented closed connected $C^{\infty}$-manifold of dimension $2m+1$ for some $m\in \N$.

\section{Review: Refined Reidemeister torsions of $\GL_n$-representations}\label{reidemeister1}
Let us briefly review algebraic torsions. Let $V$ be a $d$-dimensional vector space over a field $\F$.
For two bases of $V$, say $\mb{b}=\{b_1,\ldots ,b_d\},\mb{c}=\{c_1,\ldots , c_d\} \subset V $.
We obtain the transition matrix $P\in \GL(d;\F)$ that satisfies $b_i=\sum_{j=1}^{d}P_{(j,i)}c_j$, where $P_{(j,i)}$ is the $(j,i )$-th entry of $P$.
In what follows, we write $[\mb{b}/\mb{c}] \in \F^{\times }$ for the determinant of $P$.
Take a cochain complex
\begin{equation}\label{pp48877} C^*: 0\lra C^0 \overset{\delta^0}{\lra}C^{1}\overset{\delta^1}{\lra}\cdots\xrightarrow{ \ \delta^{k-2}
\ }C^{k-1} \xrightarrow{ \ \delta^{k-1}
\ }C^k\lra0 ,
\end{equation}
such that each $C^i$ is a finite-dimensional vector space over $\F$, and choose a basis $\mb{c}_i$ of $C^i$ for each $i$.
We employ the common notations $Z^i=\Ker(\delta^i)$, $B^i=\Im(\delta^{i-1})$, and $H^i=Z^i/B^i$.
Considering the two exact sequences
\[ 0\lra Z^i\lra C^i\overset{\delta^i}{\lra}B^{i+1}\lra 0, \quad \quad 0\lra B^i\lra Z^i{\lra}H^{i}\lra 0\]
over $\F$, we choose sections $\mathfrak{s}_i :B^{i+1} \ra C^i $ and
$\mathfrak{s}_i' :H^{i} \ra Z^i $.
For $i \leq k $, let us choose a basis $\mb{b}_i$ of $B^i$, and fix its lift to $C^{i-1}$; let $\widetilde{\mb{b}}_{i}$ denote the lift, that is, $ \widetilde{\mb{b}}_{i+1} = \mathfrak{s}_{i} (\mb{b}_{i+1 } )$.
Similarly, we select a basis $\mb{h}_i$ of $H^i$, and write $\widetilde{\mb{h}}_{i}$ for its lift in $Z^i$.
Then, under the identifications, the (ordered) union $(\mb{b}_i \cup \widetilde{\mb{h}}_{i} \cup\widetilde{\mb{b}}_{i+1})$ is a basis of $C^i$.
Denote by $\mb{c}$ and $\mb{h}$ the unions of the basis $(\mb{c}_0\cup \mb{c}_1\cup\cdots \cup\mb{c}_k)$ and $(\mb{h}_0\cup\mb{h}_1\cup\cdots \cup\mb{h}_k)$, respectively.
To summarize, {\it the torsion} of $C^*$ is defined as
$$\tau(C^*,\mb{c},\mb{h}):=\prod_{i=0}^{k}[(\mb{b}_i\cup \widetilde{\mb{h}}_{i}\cup\widetilde{\mb{b}}_{i + 1} ) / \mb{c}_i]^{(-1)^{i+1}}\in \F^{\times }.$$
It is known (see \cite{Tur}) that $\tau(C^*,\mb{c},\mb{h})$ is independent of the choices of the basis $(\mb{b}_0,\mb{b}_1,\ldots ,\mb{b}_k)$ and
the sections $\mathfrak{s}_i $ and $\mathfrak{s}_i' $.
However, the torsion depends on the choices of $\mathbf{c}_i$ and $\mathbf{h}_i$.
More precisely, if we select other such bases, $\mathbf{c}_i' $ and $\mathbf{h}_i' $, the following holds \cite[Remark 1.4.1]{Tur}:
\begin{equation}\label{pp45377} \tau(C^*,\mb{c},\mb{h})=\tau(C^*,\mb{c}',\mb{h}')\prod_{i=0}^k ([\mb{c}_i/\mb{c}'_i][\mb{h}_i'/\mb{h}_i])^{(-1)^{i+1}} \in \mathbb{F}^{\times } .
\end{equation}

Next, let us review Reidemeister torsions. Let $Y$ be a connected finite CW-complex. Take a $\GL_n$-representation $ \rho: \pi_1(Y) \ra \GL_n(\mathbb{F}) $, and regard $ \F^n$ as a left $\Z[\pi_1(Y)]$-module. Let $ \widetilde{Y}$ be the universal covering space of $Y$ as a CW-complex and $C_*(\widetilde{Y};\Z) $ be the cellular complex associated with the CW-complex. This $C_*(\widetilde{Y};\Z) $ can be viewed as a left free $\Z[\pi_1(Y)]$-module by covering transformations.
Choose a basis $ \mb{c}_i$ of the free $\Z[\pi_1(Y)]$-module $C_i (\widetilde{Y};\Z) $.
Then, we have the chain complex of the local system $C_*(\widetilde{Y};\Z) \otimes_{\Z[\pi_1(Y)] }\mathbb{F}^n $.
Dually, the cochain complex of the local system is defined by
$$C^*_\rho (Y; \mathbb{F}^n) := \mathrm{Hom}_{\mathbb{Z}[\pi_1 (Y)]\textrm{-mod}} ( C_*(\widetilde{Y};\Z), \mathbb{F}^n) .$$
As a special case, if $n = 1$ and $\rho$ is trivial, we use the notation $C^*$ instead of $C_\rho^*$.
For the standard basis $\mb{f} \subset \F^n$, 
the tensor product $ \mb{c}_i \otimes \mb{f} $ is a basis of the vector space $C_i(\widetilde{Y};\Z) \otimes_{\Z[\pi_1(Y)]} \mathbb{F}^n $.
Let $ \mb{c}_i^* =(\mb{c}\otimes \mb{f})^*$ be the dual basis of $ C^i_\rho (Y; \mathbb{F}^n) $. %
Let $\mb{c}_Y $ denote the union $\mb{c}_0^* \cup \mb{c}_1^* \cup \cdots \cup \mb{c}_{{\rm dim} Y}^* $.
Now, let us consider the subgroups
\[\begin{split} \mathrm{det}(\rho(\pi_1(Y)))&:= \{ \mathrm{det}( \rho(x)) \mid x \in \pi_1(Y) \} \subset \F^\times , \\
 \pm \mathrm{det}(\rho(\pi_1(Y)))&:= \{ \varepsilon \mathrm{det}( \rho(x)) \mid x \in \pi_1(Y) , \varepsilon \in \{ \pm 1\} \} \subset \F^\times , \end{split}\]
and the quotient group $ \mathbb{F}^{\times }/ \pm \mathrm{det}(\rho(\pi_1(Y)))$.
Furthermore, with a choice of basis $\mathbf{h}_i$ of the cohomology $H^i_\rho (Y; \mathbb{F}^n) $, {\it the Reidemeister torsion of} ($Y,\rho$) is defined as
$$ \tau(C^*_\rho(Y; \mathbb{F}^n) , \mb{c}_Y,\mb{h}) \in \mathbb{F}^{\times }/ \pm \mathrm{det}(\rho(\pi_1(Y))) .$$
As is well known \cite{Mil,Tur}, 
the torsion modulo $ \pm \mathrm{det}(\rho(\pi_1(Y))) $ does not depend on the choice of $\mb{c}_Y $.


In addition, 
let us review the sign-refined torsions by Turaev \cite{Tur, Dub}, and state Theorem \ref{ll24}. Let $ H^* (Y ;\R)$ be the ordinary cohomology over $\R $. Choosing an orientation of $\oplus_{i\geq 0} H^i(Y ;\R)$, we choose a basis $ \mathbf{h}_i^{\R} \subset H^i(Y ;\R) $ so that the sequence $(\mathbf{h}_0^{\R}, \mathbf{h}_1^{\R}, \dots)$ is a positive basis in the oriented vector space $H^*(Y;\R)$. Consider
$$ \tilde{\tau}(C^*(Y; \mathbb{R}) , \mathbf{c}, \mathbf{h}^{\R }) := (-1)^{ N(Y) }\tau(C^*(Y; \mathbb{R}) , \mathbf{c}, \mathbf{h}^{\R}) \in \R^{\times}, $$
where
\begin{equation}\label{pp45} N(Y)= \sum_{i=0}^{{\rm dim}(Y)} \bigl( \sum_{j=0}^i \mathrm{dim}H^{{\rm dim}(Y)-j}(Y ;\R) \sum_{j=0}^i \mathrm{dim}C^{{\rm dim}(Y)-j} (Y ;\R) \bigr) \in \Z /2 \Z.\end{equation}
Then, {\it the refined torsion} is defined to be
\begin{equation}\label{pp4577} \tau^0_{\rho } (Y,\mathbf{h} ) := \mathrm{sign}\bigl(\tilde{\tau}(C^*(Y; \mathbb{R}) , \mathbf{c}, \mathbf{h}^{\R})\bigr)^n \cdot \tau (C^*_\rho(Y; \mathbb{F}^n) , \mathbf{c}_Y, \mathbf{h} ) \in \F^{\times }/ \det (\pi_1 (Y)) . \end{equation}
If $n$ is even, then the sign is +1, that is, $ \tau^0_{\rho } (Y,\mathbf{h} )= \tau(C^*_\rho(Y; \mathbb{F}^n), \mathbf{c}_Y, \mathbf{h} ) $. Then, the topological invariance is shown as follows:
\begin{thm}[{
see \cite[Chapter 18]{Tur} or \cite[Chapter 2]{Dub}}]\label{ll24}
The sign-refined torsion $ \tau^0_{\rho } (Y,\mathbf{h} ) \in \F^{\times}/ \det (\pi_1 (Y))$ is independent of the order of the oriented cells of $Y$ and the choice of $\mathbf{h}^{\R}$; however, it does depend on the choice of $\mathbf{h} $. Furthermore, the torsion is invariant under simple homotopy equivalences that preserve the homology orientation.
\end{thm}

As a corollary, let us consider the situation in which $Y$ is an oriented closed $C^{\infty}$-manifold.
Recall that any two triangulations of an oriented closed $C^{\infty}$-manifold $N$ are simple homotopy equivalent (see, e.g., \cite[\S II.8]{Tur}); consequently, if $Y$ is a triangulation of $N$, the sign-refined torsion yields a topological invariant of $N$ associated with $\rho: \pi_1(N) \ra \GL_n(\F)$ and $\mathbf{h} $.

\section{Non-acyclic torsions of closed manifolds of odd dimension}\label{HG}

Since Theorem \ref{ll24} and \eqref{pp45377} imply that the refined torsion \eqref{pp4577} depends on $\mathbf{h} $,
it is reasonable to find some torsions, which depend only on the representation $\rho: \pi_1(Y) \ra \GL_n(\F)$. 
Previous studies have investigated torsions where $Y$ is an oriented closed manifold of even dimension; see, e.g.,
\cite{Sozen,Wit}.
In contrast, this paper addresses non-trivial torsions in the case $\mathrm{dim} X=2m+1$ in terms of Poincar\'{e} duality.

To this end, we prepare some cup products as follows.
Fix a bilinear form $\psi: \F^n \times \F^n \ra \F$ satisfying the {\it $\rho$-invariance}, that is,
$ \psi( \rho(g) \cdot v, \rho(g) \cdot w) = \psi(v, w)$ holds for any $ g \in \pi_1(X)$ and $ v,w \in \F^n$. Furthermore,
suppose that $ \psi$ is either hermitian or anti-hermitian.
Using the cup product $\smile$, let us consider the composite map
\begin{equation}\label{kk4} H^i_{\rho }( X;\F^n ) \otimes H^{2m+1 -i}_{\rho }( X;\F^n)\stackrel{\smile}{ \lra}
H^{2m+1}_{\rho }( X;\F^n \otimes \F^n) \xrightarrow{\ \psi_* \circ \cap [X] \ } 
\F,
\end{equation}
where $ \cap [X]$ is the cap-product with the fundamental class $[X] \in H_{2m+1}(X;\Z)$, and
the last map $\psi_*$ is the coupling with $\psi$.
In addition, we define $N(\F)$ to be $\{ y \bar{y} \mid y \in \F^{\times } \} $ as a multiplicative subgroup of $\F^{\times }$,
and we consider the quotient group
\[ \F^{\times }/ \langle N(\F),\det(\rho( \pi_1(X))) \rangle=\F^{\times }/ \{ x \bar{x} \mathrm{det}(\rho(g))\mid x \in \F^{\times } , g \in \pi_1(X) \} .\]

Then, the topological invariance of the refined torsion is shown as follows:

\begin{thm}\label{l33l24}
Let $\psi: \F^n \times \F^n \ra \F $ be a (anti-)hermitian $\rho$-invariant bilinear form.
Take a basis $\mb{h}_j $ of $H^j_\rho(X;\F^n)$ for $ j \leq m.$

Suppose the non-degeneracy of \eqref{kk4}, which defines uniquely
$ \mb{h}_{j }^{\rm dual} \subset H^{2m+1-j}_\rho(X;\F^n) $ as the dual basis of $\mb{h}_j$.
Then, the refined torsion
\begin{equation}\label{kk54} \tau^0_{\rho} (X, \{ \mathbf{h}_0 , \mathbf{h}_1 ,\dots, \mb{h}_{m}, \mathbf{h}_{m}^{\rm dual} , \mathbf{h}_{m-1}^{\rm dual} , \dots, \mb{h}_{0 }^{\rm dual} \} ) \in \F^{\times }/ \langle N(\F),\det( \rho( \pi_1(X))) \rangle\end{equation}
is independent of the choice of the bases $\mb{h}_0 ,\dots, \mb{h}_m$.

\end{thm}
\begin{proof}
Likewise the proof of Theorem \ref{ll24}, it suffices to show that the torsion is independent of the choice of $\mb{h}_j $ and the order of $\mb{h}_j $.
The independence of the order is obvious since any change in the order of $\mb{h}_j $ is parallel to that of $ \mb{h}_{j}^{\rm dual}$.
To show the remaining independence of $\mb{h}_j $, we choose another ordered basis $\mb{h}_j' $ of $H^j_\rho(X;\F^n) $. 
Since $\psi$ is (anti) hermitian , we can easily check that $ [ \mb{h}_j / \mb{h}_j' ]= \overline{ [ \mb{h}_j^{\rm dual} /(\mb{h}_j ')^{\rm dual} ] }^{-1}$.
Thus, by \eqref{pp45377}, the torsion \eqref{kk54} is equal to
\begin{equation}\label{ssg}
\tau^0_{\rho} (X, \{ \mathbf{h}_0' , \dots, \mb{h}_{m}',( \mathbf{h}_{m}')^{\rm dual} , \dots, (\mb{h}_{0 }')^{\rm dual} \} )
\prod_{j =0}^m( [ \mb{h}_j / \mb{h}_j' ]\overline{ [ \mb{h}_j / \mb{h}_j' ] } )^{(-1)^{j+1 }} \in \F^{\times }.\end{equation}
Thus, the torsion in the quotient group $\F^{\times }/ N(\F)$ is independent of the choice of the bases $\mb{h}_0,\dots, \mb{h}_m$ as required.
\end{proof}
To obtain the refined torsion, it is sensible to find a condition that ensures non-degeneracy of \eqref{kk4}.
We will give such a condition as follows:
\begin{prop}\label{l33334}
Let $\psi: \F^n \times \F^n \ra \F $ be a $\rho$-invariant bilinear form.
If $ \psi$ is non-degenerate and the field $\F$ is of characteristic zero, then there is an isomorphism $ D: H^i_{\rho }( X;\F^n ) \cong H_{\rho }^{2m+1 -i}( X;\F^n) $ and
the bilinear form \eqref{kk4} is non-degenerate.
\end{prop}
\noindent
Although this proposition may be known (see, e.g., \cite{HSW}), we defer the proof into Appendix \ref{9991}.
Here, we note that $D ( a x)= \bar{a} D(x)$ for $a \in \F$ and $ x \in H^i_{\rho }( X;\F^n ) .$
Furthermore, we will show that the refined torsion in Theorem \ref{l33l24} is 
a modification of the Poincar\'{e}-Reidemeister (scalar) product \cite{Fa, FT}; see Appendix \ref{999166} for details.

\section{Torsions as volume forms on $\SU(n)$-character varieties}\label{HG444}
In the work of Witten \cite{Wit}, the adjoint torsion of any closed surface $\Sigma$
yields a volume form on the conjugacy classes set of $\Hom(\pi_1(\Sigma), G)$,
where $G$ is a compact Lie group; see also \cite{Sozen} for higher dimensional cases of $\dim X/2 \in \Z$.
Analogously, if $\dim X = 3$, we propose a volume form of the conjugacy classes $\Hom(\pi_1(X ), G)/$conj
using the torsion \eqref{kk54}, although we restrict ourselves to special unitary representations as seen below\footnote{After the submission of this paper, 
another preprint \cite{NS} suggests another approach to define such volumes in terms of algebraic geometry.}.
Here, the restriction is due to a condition applicable to the torsion \eqref{kk54}.

Let $\SU(n)$ be the special unitary group as a subgroup of $\GL_n(\C)$.
We fix the hermitian metric on $\C^n $, which takes $(v,w) $ to $ ^t \bar{v}w$.
Furthermore, fix a Lie group $G$ and an $\SU$-representation $ \rho' : G \ra \SU (n )$.
The unitary group ensures a hermitian non-degenerate $\psi: \C^n \times \C^n \ra \C$ satisfying
the condition in Proposition \ref{l33334}.
For example, if $G$ is compact, any representation of $G$ over $\C$ is unitary.
Notice the isomorphism $\C^{\times} / N( \C) \cong \mathrm{U}(1)=S^1$ sending $[z]$ to $z/|z|$, which will appear in Example \ref{exa11}.
Thus, using Theorem \ref{l33l24}, for any homomorphism $f: \pi_1(X) \ra G $, we can define the torsion
$\tau^0_{\rho' \circ f } (X) $ in \eqref{kk54}, which lies in $\{ x \in \C \mid x \bar{x}=1 \} $.
If two homomorphisms $f$ and $f'$ are conjugate, the resulting torsions are equal by definitions.
Thus, we obtain a map
\begin{equation}\label{4ll} \frac{ \Hom ( \pi_1(X), G ) }{\mathrm{conjugate}} \lra U(1)= \{ x \in \C \mid x \bar{x}=1 \}; f \longmapsto \frac{\tau^0_{\rho' \circ f } (X)}{ | \tau^0_{\rho' \circ f } (X) | }.\end{equation}
If $G = \SU(n)$, the domain is denoted by $R_{G} (X)$ and sometimes called {\it the $\SU(n)$-character variety (of $X$)}.
Here, the topology on $R_{G}(X) $ is a quotient topology from $\Hom ( \pi_1(X), G ) $.
As an example, we compute the torsions of some Seifert 3-manifolds, when $G = \SU(2)$; see Section \ref{HG443}.
Finally, we pose a problem.
\begin{prob}\label{l835} Is the torsion function \eqref{4ll} continuous?
Moreover, if $G$ is semi-simple and contained in $\SU(m )$ and $ \rho' $ is the adjoint representation, then is the torsion function locally constant?
\end{prob}
As estimates to consider the problem, we remark on known results: first, if $\dim X$ is even, and $ \rho' $ is the adjoint representation, then the torsion function is
locally constant; see \cite{Sozen,Wit}. Meanwhile, if $X$ is a knot complement in the 3-sphere with a tori boundary, the torsions from
some adjoint representations are shown to be locally constant; see \cite{PY}. 

Furthermore, we will observe that, if $\rho$ is the adjoint representation with $G = \SU(n)$ and $\dim X = 3$, the map \eqref{4ll} can be regarded as a
volume form of an open dense subset of $R_{G}(X) $ as follows.
Let $\F$ be $\R$, and $\psi' $ be the Killing form that sends $(X,Y)$ to $\mathrm{Tr}(XY)$, where the involution is the identity. 
From the viewpoint of a stratification of $R_{G}(X) $ as a real algebraic variety,
there is an open dense set $ R_{G} (X)^{ \rm o} \subset R_{G}(X)$ that has a $C^{\infty}$-manifold structure; see, e.g., \cite{lll}.
As a classical result, recall that, for $f \in R_{G}(X)^{ \rm o}$, the first cohomology $H^1_{\rho' \circ f }(X ;\mathfrak{g})$ can be identified with
the tangent space of $ R_{G}(X)^{ \rm o}$, as in \cite{Wei,lll}.
Since the Killing form $\psi'$ is non-degenerate, the pairing \eqref{kk4} is non-degenerate by Proposition \ref{l33334}.
Thus, the correspondence
\begin{equation}\label{4l4l} H^2_{\rho' \circ f }(X ;\mathfrak{g})\lra \Hom ( H^1_{\rho' \circ f }(X ;\mathfrak{g}), \R) ; \ \ w \longmapsto (v \mapsto \psi'( (v \smile w) \cap [X] ))\end{equation}
is an isomorphism, and therefore, the second one $H^2_{\rho' \circ f }(X ;\mathfrak{g})$ can be identified with
the cotangent space of $ R_{G}(X)^{ \rm o}$. 
Let $d$ be the dimension $\dim (H^2_{\rho' \circ f }(X ;\mathfrak{g}) )$.
Thus, 
for a basis $\mb{h}_2$ of $H^2_{\rho' \circ f }(X ;\mathfrak{g})$, the $d$-fold exterior product (i.e., the determinant bundle) of $\mb{h}_2$ is a $d$-form of the manifold $ R_{G}(X)^{ \rm o}$.
Let $\mb{h}_2^{\otimes} $ denote the $d$-form, and $\mathbf{h}_2^{\rm dual}$ be the dual basis of $\mathbf{h}_2 $ by \eqref{4l4l}.
From \eqref{ssg}, the $d$-form defined to be
\begin{equation}\label{yyy}\tau_{\rho' \circ f }^0(X , \{ \mathbf{h}_2^{\rm dual},\mathbf{h}_2 \} )^{1/2} \mathbf{h}_2^{\otimes} \in  \bigwedge^{d } H^2_{\rho' \circ f }(X ;\mathfrak{g}) \setminus \{ 0\} \end{equation}
is independent of the choices of $ \mathbf{h}_2 $. In fact, if we have another $\mathbf{h}_2'$, by \eqref{pp45377} and \eqref{ssg}, we see 
\[\begin{split}
\tau_{\rho' \circ f }^0(X , \{ \mathbf{h}_2^{\rm dual},\mathbf{h}_2 \})^{1/2} \mathbf{h}_2^{\otimes} &= 
\bigr( [ \mathbf{h}_2 / \mathbf{h}_2' ]^2 \tau_{\rho' \circ f }^0(X , \{ \mathbf{h}_2^{\rm  dual},\mathbf{h}_2 \}) \bigr)^{1/2}  [\mathbf{h}_2' / \mathbf{h}_2]\mathbf{h}_2^{\otimes} \\
&= 
\bigr( [\mathbf{h}_2^{\rm  dual} / \mathbf{h}_2^{\rm  ' dual}] [\mathbf{h}_2 / \mathbf{h}_2'] \tau_{\rho' \circ f }^0(X , \{ \mathbf{h}_2^{\rm  dual},\mathbf{h}_2 \}) \bigr)^{1/2}  \mathbf{h}_2^{' \otimes} \\
&=  \tau_{\rho' \circ f }^0(X , \{ \mathbf{h}_2^{\rm ' dual},\mathbf{h}_2' \} )^{1/2} \mathbf{h}_2^{' \otimes} .\end{split}\] 
In summary, the correspondence 
\[ R_{G}(X)^{ \rm o} \lra  \wedge^{d } H^2_{\rho' \circ f }(X ;\mathfrak{g}) \setminus \{ 0\}  ; f \longmapsto \tau_{\rho' \circ f }^0(X, \{ \mathbf{h}_2^{\rm dual}, \mathbf{h}_2\} )^{1/2} \mathbf{h}_2^{\otimes} \]
can be regarded as a volume form of the variety $ R_{G}(X)^{ \rm o}$.
Thus, it is reasonable to define the volume of the $\SU(n)$-character variety as follows:
\begin{defn}\label{l8435}
Assume that $\dim(X) =3$ and $G= \SU(n)$.
We define {\it the volume of $ R_{G}(X)^{ \rm o}$} to be the integral of $\tau_{\rho' \circ f }^0(X, \{  \mathbf{h}_2^{\rm dual},\mathbf{h}_2\} )^{1/2} \mathbf{h}_2^{\otimes} $ on $ R_{G}(X)^{ \rm o}$.
\end{defn}
\noindent
In \S \ref{HG447}, we give examples of computing the volume, where $X$ is a torus bundle. 

Finally, we discuss a comparison with the preprint \cite{MP}. Here, starting from the same assumption in Definition \ref{l8435}, the authors define a volume of the variety $ R_{G}(X)^{ \rm o}$ using another approach (see Theorems 3.5 and 3.6 in \cite{MP}) and examine some relations to \cite{Wit}.
However, their definition requires a certain condition; but Definition \ref{l8435} does not.
It is seemingly interesting to check whether their volume and our volume in the condition are equal or not.



\section{Some computations of the quotient groups 
$\F^{\times }/N (\F)$}\label{HddG3}
Since the torsion \eqref{kk54} lies in the quotient group $\F^{\times }/N (\F)$ in the case $\mathrm{Im}(\rho) \subset \mathrm{SL}_n(\F)$, we provide some computations of the groups 
of some fields $\F$. 
\begin{exa}\label{exa11}\begin{enumerate}[(i)]
\item If the equation $ x \bar{x}=1$ implies $x =1 \in \F$, then $\F^{\times }/N (\F)$ is trivial.
For example, if $\bar{x} = x$ for any $x \in \F$ 
and $\F$ is quadratically closed,
then $\F^{\times }/N (\F) = 0.$ 

\item Let $q$ be a prime power, and 
$\F$ the finite field of order $q$. Suppose the involution is $\mathrm{id}_{\F}$. Since $\F^{\times }$ is the cyclic group of order $q-1$, $\F^{\times }/N (\F) \cong \Z/2$ if $q$ is odd, and
$\F^{\times }/N (\F) \cong 0 $ if $q$ is even. 
\item
Let $\F$ be $\C$, and the involution be the complex conjugacy.
Then, the map $\C^{\times } \ra S^1$ that sends $x$ to $x /|x|$ induces the isomorphism $\C^{\times }/N (\C) \cong S^1 =U(1)$.
\end{enumerate}
\end{exa}
Next, we focus on the case where
$\F$ is the fraction field of a unique factorization domain $A $ satisfying $\overline{A} \subset A$.
Let us choose a complete representative set of all prime elements of $A$ and denote the set by $\mathfrak{P}_A$.
Notice that $p \in A $ is prime if and only if $\bar{p}$ is prime.
Without loss of generality, we may suppose that every $p \in \mathfrak{P}_A$ satisfies either $\bar{p} \in \mathfrak{P}_A$ or
$ \bar{p} = a_p p$ for some $a_p \in A^{\times}$.
Let $\mathfrak{P}_A^{\Z/2} $ be the subset $\{p \in \mathfrak{P}_A \mid p^{-1} \bar{p} \in A^{\times} \}$ which is the invariant subset of the conjugation; we can choose a subset $\mathfrak{cP}_A \subset \mathfrak{P}_A $
such that $\mathfrak{P}_A =\mathfrak{P}_A^{\Z/2} \sqcup \mathfrak{cP}_A \sqcup \overline{\mathfrak{cP}_A }. $
Then, we have the isomorphism
\begin{equation}\label{4ll2} \bigoplus_{ p \in \mathfrak{P}_A^{\Z/2} } \Z/2 \oplus \bigoplus_{ q \in \mathfrak{cP}_A } \Z \cong \F^{\times}/ \langle N(\F), A^{\times } \rangle, \end{equation}
which sends $((a_p)_p,(b_q)_q) $ to $ \prod_p p^{a_p} \prod_q q^{b_q} $.
We also have a short exact sequence
\begin{equation}\label{4ll23} 0 \ra A^{\times }/\{ a \bar{a} \mid a \in A^{\times} \} \lra
\F^{\times}/ N(\F)
\lra \F^{\times}/ \langle N(\F), A^{\times } \rangle \ra 0, \end{equation}
where the first (resp. second) map is obtained by the inclusion $A \hookrightarrow \F$ (resp. the projection).
In summary, from \eqref{4ll2} and \eqref{4ll23}, we can sometimes determine the group $\F^{\times }/N (\F) $ as follows:
\begin{exa}\label{exa143}Suppose that the involution on $\F$ is the identity. 
Then, $ \mathfrak{cP}_A $ is obviously empty, and $p= \bar{p}$ for any $p \in \mathfrak{P}_A $; thus, the sequence \eqref{4ll23} splits. Hence,
\[ \F^{\times }/N (\F) =\F^{\times }/(\F^{\times })^2\cong (A^{\times} /(A^{\times})^2 ) \oplus \bigoplus_{p \in \mathfrak{P}_A } \Z/2.\]
For example, if $\F= \Q$ and $A=\Z$, then $\Q^{\times }/N (\Q) \cong \bigoplus_{p \in \mathrm{Spec}(\Z) } \Z/2$.
Furthermore, if $\F= \R $, then $\R^{\times }/N (\R) \cong \Z/2$.
\end{exa}

\begin{exa}\label{exa1983}
As mentioned in Section \ref{HG444}, let $\F$ be the fraction field $\Q(t_1, \dots,t_{\ell} )$ over $\Q$.
Here, the involution of $f(t_1, \dots,t_{\ell} ) \in \F$ is defined by $f(t_1^{-1}, \dots,t_{\ell}^{-1})$.
Notice that $\F $ is the fraction field of $A= \Q[t^{\pm 1}_1 ,\dots, t_{\ell}^{\pm 1 }]$ as a unique factorization domain.
We say $f(t_1, \dots,t_{\ell} ) \in \F$ to be {\it reciprocal} if $\bar{f}= r t_1^{n_1}\cdots t_{\ell}^{n_{\ell}} f $ for some
$r \in \Q^{\times}$ and $ n_i \in \Z$.
Then, the subset $\mathfrak{P}_A^{\Z/2} $ is the set consisting of reciprocal irreducible polynomials in $\Q[t^{\pm 1}_1 ,\dots, t_{\ell}^{\pm }]$,
and any $ p \in \mathfrak{cP}_A $ is not reciprocal.
It is not so hard to check that the group $ A^{\times }/\{ a \bar{a} \mid a \in A^{\times} \} $ is isomorphic to
$ \Q^{\times}/(\Q^{\times})^2 \oplus \Z^{\ell} $, where $\Z^{\ell}$ is generated by $t_1,\dots, t_{\ell}$;
the exact sequence \eqref{4ll23} turns out to be
\begin{equation}\label{54ll233}0 \lra \Q^{\times}/(\Q^{\times})^2 \oplus \Z^{\ell} \lra \Q(t_1, \dots,t_{\ell})^{\times} / N(\Q(t_1, \dots,t_{\ell})) \ra \end{equation}
\[ \qquad \ra \bigoplus_{p \in \mathfrak{P}_A^{\Z/2} } \Z/2 \oplus \bigoplus_{q \in \mathfrak{cP}_A } \Z \lra 0 . \]
Here, let us notice that $(g-\bar{g} )^2= - (g-\bar{g})\overline{(g-\bar{g})}= -1 \neq 1 $ hold in $ \F^{\times}/N(\F)$ for any $g \in A$ such that $g \neq \bar{g}$;
Thus, $\F^{\times}/N(\F)$ contains many elements of order 4, unless $\ell = 1$.
In particular, the exact sequence does not split.

We proceed with the computation for the simplest case of $\ell =1$.
Let $\mathfrak{sP}_A $ be the set of symmetric irreducible polynomials in $\Q[t^{\pm 1}]$.
Since $\Q[t]$ is a Euclidean domain,
any polynomial $f \in \Q[t^{\pm 1}] $ ensures some $n,m,u,s \in \Z,r \in \Q, a_i \in \mathfrak{P}_A^{\Z/2}$ and $b_j \in \mathfrak{P}_A$ such that
\[ f=r (1-t)^u (1+t)^s a_1 \cdots a_n b_1 \cdots b_m, \qquad \textrm{and} \qquad \overline{a_i}=a_i.\]
Since
\[ (1-t)^2= -t(1-t)(1-\bar{t}), \qquad \textrm{and} \qquad (1+t)^2= t(1+t)(1+\bar{t}),\]
the subgroup of $ \Q(t)^{\times}/N(\Q(t) )$ generated by $-1,t \pm 1, t$ is isomorphic to
$\Z \oplus \Z/4$, where the summand $\Z$ is generated by $t+1$ and $\Z/4$ is done by $(t-1)(1+t^{-1})$.
In conclusion, we obtain from \eqref{54ll233} an isomorphism
\[ \Q(t)^{\times}/N(\Q(t) ) \cong \Z \oplus \Z/4 \oplus \bigoplus_{p \in \mathrm{Spec}(\Z) \setminus \{0\} } \Z/2 \oplus
\bigoplus_{p \in \mathfrak{sP}_A \setminus \{ t \pm 1 \} } \Z/2 \oplus \bigoplus_{q \in \mathfrak{cP}_A } \Z . \]

\end{exa}
\begin{exa}\label{exa1395}
Suppose that $\F \subset \C$ is a number field and the involution is the complex conjugacy satisfying $\overline{\F} \subset \F$.
Let $\mathcal{O}_{\F} \subset \F$ be the integral closure of $\Z$.
Dirichlet's unit theorem implies that the first term in \eqref{4ll23} is computed as
\[ \mathcal{O}_{\F}^{\times} / \{ a \bar{a} \mid a \in \mathcal{O}_{\F}^{\times} \} \cong
(\Z/2\Z)^{r_1-1} \oplus \Z^{r_2} \oplus (\{ a \in \F \mid a \bar{a}=1 \}/(\Z/2\Z)), \]
where $ r_1$ is the number of real embeddings and $r_2$ is the number of conjugate pairs of complex embeddings of $\F$, and the latter quotient by $\Z/2\Z$ is defined by the involution.
Meanwhile, recall that the ideal class group $C(\mathcal{O}_{\F} )$ is trivial if and only if $\mathcal{O}_{\F} $ is a UFD.
Thus, if $ C(\mathcal{O}_{\F} )$ is trivial, then we can
compute the group $\F^{\times }/N (\F) $ by \eqref{4ll2} and \eqref{4ll23}.

However, conversely, if $ C(\mathcal{O}_{\F} ) \neq 1 $, it seems hard to obtain a similar isomorphism to \eqref{4ll2} because of the following reasons.
Hasse's local-global principle implies that the computation of $ \F^{\times }/N (\F) $ is derived from
those of $\F_{\mathfrak{p}}^{\times }/(N (\F_{\mathfrak{p}})) $, where $\mathfrak{p} \in \mathrm{Spec}( \mathcal{O}_{\F})$ and $\F_{\mathfrak{p}} $ is the local field.
However, even if the residue field of $\mathfrak{p} $ is of characteristic two, the computation of $\F_{\mathfrak{p}}^{\times }/(N (\F_{\mathfrak{p}})) $ is difficult.
Furthermore, if $\Q \subset \F $ is not an abelian extension, in general,
it is also difficult to determine all prime ideals of $\mathcal{O}_{\F} $ even using the class field theory.
\end{exa}

\section{Some topological invariants from refined torsions}\label{HG33}
As corollaries of Theorem \ref{l33l24}, we introduce some topological invariants in terms of refined torsions.
For example, we show that the theorem produces extended definitions of
abelian Reidemeister torsions and (multivariable) Alexander polynomials of links; see Sections \ref{HG4464} and \ref{HG43344}.
Before those, we also discuss some torsions with trivial coefficients, which are shown to be not so interesting (Sections \ref{HG444j44} and \ref{HG42442}).


\subsection{Torsions from the trivial coefficients over a field $\F$}\label{HG444j44}
We begin by examining the case where $\rho$ with $n = 1$ is trivial and $\psi$ is the canonical multiplication of $\F $.
Then, $H^j(X;\F)$ is the ordinary cohomology of $X$, and the ordinary cup product in \eqref{kk4} is widely known to be non-degenerate.
\begin{defn}\label{l339l24} Let $X$ be an oriented closed manifold of dimension $2m+1$.
We define the torsion of $X$ to be the refined torsion \eqref{kk54} in $\F^{\times }/(\F^{\times } )^2$.
\end{defn}

\begin{exa}\label{l366l24} Suppose that $X$ admits a CW-complex structure that
all boundary maps of the cellular complex are zero.
Then, $ \mathbf{h}_{j} = \mathbf{c}_{j} $, and $ \mathbf{b}_{j}= \emptyset$ by definitions.
Thus, to compute the refined torsion, we may only find a dual basis $\mathbf{h}_{j}^{\rm dual} $ by computing the cup product
and $[ \mathbf{h}_{j}^{\rm dual} / \mathbf{c}_{2m+1-j} ] \in \F^{\times }$.
In particular, the refined torsion is a weaker invariant than the cup product structure of $X$ in coefficients $\F$.

However, we give two examples.
Let $X$ be the torus $(S^1)^{2m+1}$. Consider the CW-structure of $S^1$ consisting of a $1$-cell and a $0$-cell and the product CW-structure of $X$. 
Then, by an easy computation of the cross product, we can check that $ \mathbf{h}_{j}^{\rm dual} = \mathbf{c}_{2m+1-j} $.
Hence, the torsion is trivially 1. 

On the other hand, let $\F $ be the finite field $\F_p$ and $X$ be the 3-dimensional lens space $L(p,q)$, where $p$ is an odd prime and $q \in \Z$ is relatively prime to $p$.
Then, by the genus-one Heegaard decomposition, we have a CW-structure of $ X$ such that
the $k$-cell is single with $ 0 \leq k \leq 3$, and that the boundary map is zero.
Then, by the resulting computation of the cup product (see, e.g., \cite[Example 5.1]{Nos2}), $ \mathbf{h}_{1}^{\rm dual} = q \mathbf{c}_{2} $ and
$ \mathbf{h}_{0}^{\rm dual} = \mathbf{c}_{3}$. In conclusion, the refined torsion is $q \in \F_p^{\times }/(\F_p^{\times } )^2\cong \Z/2 .$
In particular, the torsions distinguish between $ L(7,1)$ and $L(7,2)$.
However, as seen above, the torsion in this subsection is a weak topological invariant.

\end{exa}

\subsection{Torsions from the trivial coefficients over $\Q$}\label{HG42442}
We focus on the rational case $\F=\Q$ and suggest a
modification of Definition \ref{l339l24}.
For a finitely generated abelian group $A$, we denote by $\mathrm{Free}(A)$ the free submodule of $A$.
Notice that the inclusion $\iota : \Z \hookrightarrow \Q$ induces an isomorphism
$\iota_* : H^*(X;\Z) \otimes \Q \cong H^*(X;\Q)$, and the cup product
\[\smile: \mathrm{Free}( H^j(X;\Z) ) \times \mathrm{Free}( H^{2m+1-j}(X;\Z) ) \lra H^{2m+1}(M;\Z) \cong \Z\]
is unimodular, that is, the determinant of $\smile$ is $+1$.
\begin{prop}\label{l3l}
Choose a basis $\mb{h}_j $ of $\mathrm{Free}( H^j(X;\Z) ) $ for $ j \leq m $.
Let $ \mb{h}_{j}^{\rm dual} \subset H^{2m+1-j}(X;\Z) $ be the dual basis of $\mb{h}_j$.
The refined torsion $\tau^0_{\rho} (X, \mb{h} ) \in \Q$ in
\eqref{pp4577} is a topological invariant of $X$.
\end{prop}
Since the proof is almost the same as the proof of Theorem \ref{l33l24}, we omit the details.
However, if $X$ is a rational homology sphere or if $\dim X=3$, then the torsion is not entirely interesting as follows: 
\begin{prop}\label{l3kl}
\begin{enumerate}[(i)]
\item
If $H^*(X;\Q) \cong H^*(S^{2m+1};\Q)$, then
the torsion $\tau^0_{\rho} (X, \mb{h} )$ is equal to $ \prod_{i=1}^{2m} | H^{i} (X;\Z )|^{ (-1)^i } \in \Q$.
\item If $m = 1$, then the torsion $\tau^0_{\rho} (X, \mb{h} )$ equals $\varepsilon_X | H^{1} (X;\Q/ \Z )| ^{ -1} \in \Q$ for some $ \varepsilon_X \in \{ \pm 1 \}$.
\end{enumerate}
\end{prop}
\begin{proof}By Morse theory, we can choose a CW-complex on $X$ such that
the $0$-th and $(2m+1)$-th cells are single.
Then, the coboundary maps $\delta^0$ and $\delta^{2m}$ in the cellular complex $ C^*(X;\Z)$ are zero.
Consider the subcomplex $ E^*$ defined by $C^*(X;\Z) $ if $1 \leq * \leq 2m$ and by zero otherwise.
If $H^*(X;\Q) \cong H^*(S^{2m+1};\Q)$, then $E^* \otimes \Q $ is acyclic, and
$\tau^0_{\rho} (X, \mb{h} )$ is equal to the torsion of $E^* \otimes \Q $;
Theorem 1.4 in \cite{Tur} immediately implies that the torsion of $E^* \otimes \Q $ is equal to $ \prod_{i=1}^{2m} | H^{i} (X;\Z )|^{ (-1)^i } \in \Q$. Hence, we complete the proof of (i).

Next, consider the case $m = 1$, i.e., $\dim X = 3.$
Then, the subcomplex is presented as $0 \ra C^1(X;\Z)\stackrel{\delta^1}{ \lra }C^2(X;\Z) \ra 0.$
Suppose that the elementary divisors of $\delta^1$ are $e_1, \dots, e_k, 0,\dots, 0$ for some $e_i \in \mathbb{N}$.
Note that $|H^{1} (X;\Q/ \Z )|= |H_2 (X;\Z )/ \mathrm{Free} (H_2 (X;\Z ) )|=e_1 e_2 \cdots e_k. $
The elementary divisor theory implies that $\delta^1$ is the diagonal matrix $ e_1 \oplus \cdots \oplus e_k \oplus 0 \oplus \cdots \oplus 0$ up to the choice of the base of $C^2(X;\Z) $. Since $Z^1 =H^1$, we can verify that the torsion of the diagonal matrix is
equal to $ \varepsilon_X (e_1 \cdots e_k)^{-1}$. In summary,
the required torsion is equal to $ \varepsilon_X | H^{1} (X;\Q/ \Z )| ^{ -1} $ as claimed in (ii). 
\end{proof}


\subsection{Non-acyclic abelian Reidemeister torsions}\label{HG4464}
The purpose of this subsection is to show that the torsions in Theorem \ref{l33l24} in abelian coefficients yield
an extended definition of the abelian Reidemeister torsions including non-acyclic cases.

For this purpose, let us first review abelian Reidemeister torsions from \cite[Chapter II]{Tur}.
Choose the decomposition
\[\mu: H_1(X;\Z) \cong \mathrm{Free} (H_1(X)) \oplus \mathrm{Tor} (H_1(X)) \]
and a generating set $\{ t_1, \dots, t_\ell \} $ of $\mathrm{Free} (H_1(X)) $.
Let $\rho_{\rm ab}$ be the composition of the abelianization of $\pi_1(X)$ with $\mu$ and the map $\mathrm{Free} (H_1(X)) \ra \GL_1(\F) $ that sends $ t_i $ to $t_i$ itself, where $\F$ is the field $\Q(t_1, \dots, t_\ell )$ as in Example \ref{exa1983},
which is a local coefficient via $\rho_{\rm ab}$.
If the complex $C^*_{\rho}(X;\F)$ 
is acyclic,
then the torsion of $C^*_{\rho}(X;\F) $ is called {\it the maximal abelian torsion},
and it is well-studied together with its relationship to the Alexander polynomial and Milnor torsions; see \cite{Mil, Tur}.

On the other hand, concerning the torsion parts, it is known (see \cite[Theorem 12.5]{Tur}) that the group ring $\Q [\mathrm{Tor} (H_1(X)) ]$ over $\Q$ is
isomorphic to a direct sum of some cyclotomic fields $\F_1 \oplus \cdots \oplus \F_m$,
where $ \F_i =\Q(\mu_i) \subset \C$ for some root of unity $\mu_i$. 
Let $ p_i: \Q [\mathrm{Tor} (H_1(X)) ] \ra \F_i$ be the projection.
Then, the composition of the abelianization of $\pi_1(X)$ and $p_i$ gives a representation $\rho_i: \pi_1(X) \ra \F_i^{\times} =\GL_1(\F_i)$.

We will explain the above purpose in detail.
For $f= f(t_1,\dots, t_{\ell}) \in \F $, we define the involution $\bar{f}$ as $f(t_1^{-1},\dots, t_{\ell}^{-1}) $, and
similarly, the involution on $ \Q [\mathrm{Tor} (H_1(X)) ] $ by the $\Q$-linear extension of the map $g \mapsto -g$ on $\mathrm{Tor} (H_1(X))$.
Define $\psi: \F \times \F \ra \F$ and
$\psi_i: \F_i \times \F_i \ra \F_i $ by $(a,b) \mapsto a \bar{b}$, which satisfy the non-degeneracy in Proposition \ref{l33334}.
Therefore, as a result of Theorem \ref{l33l24}, even if $C^*_{\rho}(X;\F)$ 
is not acyclic, we can define the torsion
\begin{equation}\label{pop1} \tau_{\rho_{\rm ab} } (X) \in \F^{\times }/ \langle N(\F),t_1^{\pm 1}, \dots, t_{\ell}^{\pm 1}\rangle ,\end{equation}
similar to \eqref{kk54}.
On the other hand, for the torsion subgroup of $H_1(X;\Z) $,
we can define the torsion to be the tuple 
$$ \bigl(\tau_{\rho_1 }(X , \mathbf{h}) , \dots , \tau_{\rho_m }(X , \mathbf{h})\bigr) \in \F_1^{\times} /\langle \mu_1, N(\F_1) \rangle\times \cdots \times \F_m^{\times} / \langle \mu_m, N(\F_m) \rangle, $$
as in \cite[Section II. 13]{Tur}.
In conclusion, these torsions in non-acyclic cases can be interpreted as an extension of the abelian Reidemeister torsions. 

\subsection{Extension of the Alexander polynomial of links}\label{HG43344}
Next, we define torsions of any link, of which the Alexander polynomial is zero, as a generalization of the Alexander polynomials of links.

For this, let us review the Alexander polynomial; see, e.g., \cite[\S II. 15]{Tur}.
Let $L \subset S^3$ be a link of $\ell$-components. 
Fix an abelianization $\mu : \pi_1(S^3 \setminus L) \ra \Z^\ell $ and
a generating set $\{ t_1, \dots, t_{\ell } \} $ of $\Z^\ell $.
Let $\F$ be the fractional field $\Q(t_1, \dots, t_{\ell})$. 
Since $\mu$ is regarded as a representation $\pi_1(S^3 \setminus L) \ra \GL_1(\Z [t_1^{\pm 1 }, \dots, t_\ell^{\pm 1 } ] )$ that sends $g$ to $\mu(g)$,
we have the homology $H_* ( S^3 \setminus L ; \Z [t_1^{\pm 1 }, \dots, t_\ell^{\pm 1 } ] ) $ of local systems.
Then, {\it the (multivariable) Alexander polynomial}
\[\Delta_L(t_1, \dots, t_{\ell}) \in \Z [t_1^{\pm 1 }, \dots, t_\ell^{\pm 1 } ]/ \{ \pm t_{1}^{k_1} \cdots t_{\ell }^{k_\ell } \mid (k_1,\dots, k_{\ell}) \in \Z^{\ell}\} \]
is defined to be
the order of $H_1 ( S^3 \setminus L ; \Z [t_1^{\pm 1 }, \dots, t_\ell^{\pm 1 } ] ) $, that is, the minimal polynomial that
annihilates $H_1 ( S^3 \setminus L ; \Z [t_1^{\pm 1 }, \dots, t_\ell^{\pm 1 } ] ) $.
In the case $t= t_1 = \cdots =t_{\ell}$, the one variable polynomial $\Delta_L( t, \dots, t) \in \Q(t)/\{ t^k\}_{k \in \Z} $ is also called {\it the Alexander polynomial}.
It is known that the polynomial $\Delta_L(t_1, \dots, t_{\ell})$ is not zero if and only if the homology $H_1 ( S^3 \setminus L ; \F) $ vanishes,
and $\Delta_L(t, \dots, t)$ is not zero if and only if the homology $H_1 ( S^3 \setminus L ; \Q(t)) $ vanishes.
If $\ell = 1$, that is, $L$ is any knot, then $\Delta_L (t )\neq 0 $ is well known.

We also review \eqref{p7p7} below, which is shown by Turaev.
Let $M_L$ be the closed 3-manifold obtained by the zero-surgery on $L$. 
Then, the abelianization $\mu$ induces an abelianization $\pi_1(M_L) \ra \Z^\ell$, which is considered as a representation
\[ \mu_{\rm ab}:\pi_1(M_L) \ra \GL_1(\F) = \GL_1(\Q(t_1, \dots, t_{\ell})).\]
Thus, $\F= \Q(t_1, \dots, t_{\ell}) $ can be regarded as local coefficients of $M_L.$ 
It is known (see \cite[Corollary 15.2 and Theorem 17.2]{Tur}) that, if the Alexander polynomial $\Delta_L(t_1, \dots, t_{\ell}) $ is not zero, then
\begin{equation}\begin{split}\label{p7p7} \tau_{\mu_{\rm ab}}^0 (M_L)= \Delta_L(t_1, \dots, t_{\ell}) &\prod_{j: 1 \leq j \leq \ell} (t_j -1 ) \\ &\in \F^{\times }/ \{ r t_{1}^{k_1} \cdots t_{\ell }^{k_\ell } \mid (k_1,\dots, k_{\ell}) \in \Z^{\ell}, r \in \Q^{\times} \},\end{split}\end{equation}
and that the same equality with $t=t_1 = \cdots =t_{\ell}$ holds if $\Delta_L(t, \dots, t) \neq 0$. 

In summary, it seems sensible to extend the Alexander polynomial from the viewpoint of refined torsions. 
\begin{defn}\label{l3424} Suppose $\mu >1$.
Let $\F$ be $ \Q(t_1, \dots, t_{\ell})$ as above.
If $\Delta_L(t_1, \dots, t_{\ell})$ is zero, we define {\it an (extended) Alexander polynomial}
of $L$ to be the torsion
\begin{equation}\label{p7p74} \tau^{0}_{\mu_{\rm ab}} (M_L )\in \F^{\times}/ \langle N(\F) , \{ t_{1}^{k_1} \cdots t_{\ell }^{k_\ell } \mid (k_1,\dots, k_{\ell}) \in \Z^{\ell}\} \rangle, \end{equation}
where $\tau^{0}_{\mu_{\rm ab}} $ is defined in \eqref{pop1}. 
\end{defn}
Recall that the quotient group in \eqref{p7p74} is computed in Example \ref{exa1983}.
However, as a future problem, it remains to establish a procedure for computing the torsions of links with Alexander polynomial zero.

\section{Further refinements of the torsions with $b_1\geq 1$}\label{HG99}
In the preceding sections, we defined the torsions modulo the subgroup $N(\F)$. This section restricts to 3-manifolds and suggests further refinements of the torsion without using $N(\F)$; 
see Propositions \ref{l32l245} and \ref{l32l2451}.
Throughout this section, we suppose a closed 3-manifold $X$ and
an epimorphism $\alpha: \pi_1(X) \ra \Z$.
In particular, the first Betti number of $X$ is larger than zero.
By Poincar\'{e} duality $H_2( X;\Z) \cong H^1( X;\Z) \cong \Hom( \pi_1(X),\Z) $, we can choose a homology 2-class $\Sigma_{\alpha} \in H_2( X;\Z)$
as a dual of $\alpha$. 

\begin{defn}\label{l313l24} As in Theorem \ref{l33l24}, take a representation $ \rho:\pi_1(X) \ra \GL_n(\F) $
and a non-degenerate (anti-)hermitian $\rho$-invariant bilinear form $\psi: \F^n \times \F^n \ra \F $.
Then, the triple $( \rho, \psi, \Sigma)$ is said to be {\it admissible}, if the following pairing and \eqref{kk4} are non-degenerate:
\begin{equation}\label{p6p6} H^1_{\rho }( X;\F^n ) \otimes H^{1}_{\rho }( X;\F^n)\stackrel{\smile}{ \lra}
H^{2}_{\rho }( X;\F^n \otimes \F^n) \xrightarrow{\ \cap [\Sigma_{\alpha}] \ } 
H_0^{\rho }( X;\F^n \otimes \F^n)\stackrel{\psi_*}{ \lra} \F. \end{equation}
\end{defn}

\begin{prop}\label{l32l245}
Suppose that the triple $ ( \rho, \psi, \Sigma)$ is admissible, 
and the zeroth homology $H^\rho_0( X;\F^n )$ vanishes.
By the non-degeneracy, let us choose a normalized orthonormal basis $\mb{h}_1$ of $H^1_{\rho }( X;\F^n ) $ with respect to \eqref{p6p6},
and the dual basis $\mb{h}_1^{\rm dual}$ of $H^2_{\rho }( X;\F^n ) $.
Then, the refined torsion
\begin{equation}\label{kk534} \tau^0_{\rho} (X, \{ \mb{h}_{1}, \mathbf{h}_{1}^{\rm dual} \}) \in \F^{\times }/ \pm \det(\pi_1(X)) \end{equation}
modulo $\pm \det(\pi_1(X)) $ is independent of the choice of the basis $\mb{h}_1$.
\end{prop}
\begin{proof}
Consider another orthonormal basis $\mb{h}_1'$. Then,
there is an element $T$ of the unitary group $ U( \mathrm{dim} H^1( X;\F^n ) ;\F )$ such that $T \mb{h}_1= \mb{h}_1'$.
Since $ [ \mb{h}_1 / \mb{h}_1' ] =\mathrm{det}(T) \in N(\F)$ and $\mathrm{det}(T) \overline{\mathrm{det}(T) }=1$,
as seen in the proof of Theorem \ref{l33l24}, $\mathrm{det}(T)= [ \mb{h}_1/ \mb{h}_1' ]= \overline{ [ \mb{h}_{1}^{\rm dual} /(\mb{h}_{1} ')^{\rm dual} ] }^{-1}$.
Hence,
\[ \tau^0_{\rho} (X, \{ \mb{h}_{1}, \mathbf{h}_{1}^{\rm dual} \} )=
\tau^0_{\rho} (X, \{ \mb{h}_{1}',(\mathbf{h}_{1}')^{\rm dual}\} ) \frac{[ \mb{h}_1/ \mb{h}_1' ]}{ [ \mb{h}_{1}^{\rm dual} /(\mb{h}_{1} ')^{\rm dual} ]}= \tau^0_{\rho} (X, \{ \mb{h}_{1}', (\mathbf{h}_{1}')^{\rm dual} \}) \]
holds in $\F^{\times }/\pm \det(\pi_1(X))$. This implies the required independence.
\end{proof}
However, in general, there is no clarity on whether the existence of such a 2-class $\Sigma_{\alpha}$ is admissible; see \S \ref{HG443} for examples of such existences.
For example, even if $X$ is a fibered 3-manifold, it does not always hold that the canonical map $\alpha: \pi_1(X)\ra \Z$ gives an admissible pairing \eqref{p6p6}; see Proposition \ref{lli} for details. Alternatively, we suggest another refinement.
\begin{prop}\label{l32l2451}
Assume that the inclusion $\Sigma_{\alpha} \hookrightarrow X $ induces an isomorphism
$ H^2_\rho(X ;\F^n ) \cong H^2_\rho( \Sigma_{\alpha} ;\F^n ) \cong \F $ which is generated by the fundamental class $[\Sigma_{\alpha} ]$,
and that the pairing \eqref{kk4} is non-degenerate as above.
Then, the following torsion depends only on $ ( \rho, \psi, \alpha) $:
\begin{equation}\label{kk5344} \tau^0_{\rho} (X, \{ [\Sigma_{\alpha} ]^{\rm dual} , [\Sigma_{\alpha} ]\}) \in \F^{\times }/ \pm \det(\pi_1(X)). \end{equation}
\end{prop}
Since the proof can be obtained easily, we omit the details of the proof.
The assumptions in the above propositions are strong in many cases; however, we later see that
the refinement is of use for the adjoint torsions of the torus bundles; see Section \ref{HG447}.

\section{Some computations of the refined torsions of 3-manifolds}\label{HG2}
In general, the computation of the refined torsions seems not easy. In this section, we focus on closed 3-manifolds and develop a 3-dimensional procedure for
computing the refined torsions. That is, we fix $m = 1 $ and $\dim X = 3.$

Section \ref{HG222} reviews Heegaard splittings and cup products in terms of identities,
and Sections \ref{HG443} and \ref{HG447} compute some refined torsions.

\subsection{Heegaard splittings and the presentation of the fundamental group}\label{HG222}
Let $X$ be an oriented closed 3-manifold, and take a representation $\rho: \pi_1 (X) \ra \GL_n (\F)$.
In this subsection, following \cite{TrotterAM62}, we establish a procedure to compute cup products.
There is nothing new in this subsection.

We choose a genus-$g$ Heegaard splitting of $X. $
Then, $X$ consists of a single 0-handle, $g$ pieces of 1-, 2-handles, and a single 3-handle.
Thus, using the van Kampen theorem, we have a group presentation $\langle x_1,\dots, x_g \mid r_1, \dots, r_g \rangle$ of $\pi_1(X)$.
Moreover, the (augmented) cellular complex of the universal covering space of $X$ can be described as
\begin{equation}\label{z0} C_*(\widetilde{X} ;\Z): 0 \ra \Z[\pi_1(X)] \stackrel{\partial_3}{\lra} \Z[\pi_1(X)]^g \stackrel{\partial_2}{\lra}\Z[\pi_1(X)]^g \stackrel{\partial_1}{\lra}\Z[\pi_1(X)] \ra \Z \ra 0 ,\end{equation}
where $\Z[\pi_1(X)]$ is the group ring of $\pi_1(X)$.
It is known (see, e.g., \cite{TrotterAM62}) that the boundary maps $\partial_1$ and $\partial_2 $ can be represented as
\begin{align}\label{z1} \partial_2&= \bigl\{ [\frac{ \partial r_j}{\partial x_i}] \bigr\}_{ 1 \leq i,j \leq g} \in \mathrm{Mat}(g \times g; \Z[ \pi_1(X) ]), \\
\label{z2} \partial_1&=( 1-x_{1}, 1-x_{2}, 1-x_3, \dots, 1-x_{g})^{\rm transpose}. \end{align}
Here, $\frac{ \partial r_j}{\partial x_i} $ is the Fox derivative of $r_j$ with respect to $x_i$.

Next, we review a description of the third map $\partial_3$.
Let $F$ be the free group $\langle x_1, \dots, x_g\ | \ \rangle$ and $P$ be the free group $ \langle \rho_1,\ldots,\rho_g \ | \ \rangle$.
We define the group homomorphism $\Psi: P *F \ra F$ by $\Psi (\rho_j)=r_j$ and $\Psi(x_i)=x_i.$
An element $s \in P*F$ is {\it an identity} if $s \in \Ker(\Psi)$ and $s$ can be written as $\prod_{m=1}^\ell w_m \rho_{j_m}^{\epsilon_m} w_m^{-1}$ for some $w_m \in F, \ \epsilon_m \in \{\pm 1 \}, $ and $\rho_{j_m}\in P$.
\begin{thm}[{\cite{SieradskiMZ80}; see also \cite[Section 4]{Wakijo}}]\label{hdihai2}
For any closed 3-manifold $X$, there exists an identity $W_X$ such that $ \partial_3 =\bigl( \mu \bigl( [\psi (\frac{\partial W_X}{\partial \rho_1})]\bigr), \dots,  \mu \bigl([\psi (\frac{\partial W_X}{\partial \rho_g})]\bigr) \bigr) $. 
Here, $\mu$ is the natural surjection from $F$ to $\pi_1(X)$.
\end{thm}
Furthermore, in terms of the identity $W_X$, we will give the formula to determine the pairing \eqref{kk4}. 
Let $c_3$ denote the unit $1 $ in $ C_3(\widetilde{X} ;\Z)=\Z[\pi_1(X)] $.
If the identity $W_X$ is written as $\prod_{k=1}^\ell \omega_k \rho_{j_k}^{\epsilon_k} \omega_k^{-1}$, we define
\begin{equation}\label{b4} D^{\sharp}(c_3) = \sum_{k=1}^\ell \epsilon_k \bigl( \sum_{i=1}^g [ \frac{\partial \omega_k}{\partial x_i}] a_i \otimes \omega_k b_{j_k} \bigr) \in C_1(\widetilde{X};\Z) \otimes C_2(\widetilde{X};\Z) . \end{equation}
Then, for cochains $p \in C^1_\rho ( X;\F^n )$ and $q \in C^2_\rho ( X;\F^n)$, we define a 3-cochain $ p \smile q$ by
\[ p \smile q (u c_3 ):= (p \otimes q)(u D^{\sharp} (c_3)) \in \F^n \otimes_{\F} \F^n. \]
Here, $u \in \Z[\pi_1(X)]$. For a $\rho$-invariant bilinear map $\psi : \F^n \otimes \F^n \ra \F$, the map
\begin{equation}\label{b5}
\smile : C^1_\rho(X ;\F^n) \otimes C^2_\rho(X;\F^n) \lra C^3_\rho(X ;\F ); \ \ (p,q) \longmapsto \psi (p \smile q),
\end{equation}
induces the bilinear map on cohomology. Then, as a result of \cite[\S 2.4]{TrotterAM62},
the bilinear map induced by \eqref{b5} is known to be equal to the pairing \eqref{kk4} with $(m, i)=(1,1)$.

In addition, we review a procedure for computing the cup product \eqref{p6p6} from \cite[\S 2.4]{TrotterAM62}.
We define $\alpha : F \ra C_1(\widetilde{X} ;\Z)$ and $\kappa : F \times F \ra C_1(\widetilde{X} ;\Z) \otimes C_1(\widetilde{X} ;\Z) $ by
\begin{equation*}
\alpha(w) :=
\begin{pmatrix}
\mu(\dfrac{\partial w}{\partial x_1}) &  \cdots & \mu(\dfrac{\partial w}{\partial x_g})
\end{pmatrix}^{\rm transpose}, \hspace{0.5cm}
\kappa(u,v) := \alpha(u) \otimes \mu(u) \alpha(v).
\end{equation*}
Then, we can check that $\kappa$ is a 2-cocycle of $F$.
Since $H^2(F;C_1(\widetilde{X} ;\Z) \otimes C_1(\widetilde{X} ;\Z))=0$ by the freeness of $F$, we have uniquely
a map $ \Upsilon : F \to C_1(\widetilde{X} ;\Z) \otimes C_1(\widetilde{X} ;\Z)$ defined by the following rules:
\[
\Upsilon(uv)=\Upsilon(u)+\mu(u)\Upsilon(v)+\kappa(u,v), \hspace{0.5cm}
\Upsilon(1) = 0,\hspace{0.5cm}
\Upsilon(x_{i})=0
\]
for any $u,v \in F$ and $i \in \{1, \ldots, g\}$.
As is known, for any 1-cocycles $ f, f' \in C^1_\rho(X ;\F^n)$ and 2-cycle $c = (c_1, c_2, \ldots, c_g) \in \Z \otimes C_2(\widetilde{X} ;\Z)$,
\begin{equation}\label{popo}
\psi_* ( (f \smile f') \cap c)
=
\sum_{i=1}^{g} \psi \circ (f\otimes f^{\prime}) (c_i \Upsilon(r_i)) \in \F.
\end{equation}
Consequently, when we can describe the 2-chain $c \in \Z \otimes C_2(\widetilde{X};\Z)$ concretely, we can compute 
the pairing \eqref{p6p6} as required.

\subsection{Computation 1; some Seifert 3-manifolds}\label{HG443}

In this section, for three integers $(n,m,\ell) \in \mathbb{N}^3 $ with $\ell \geq m \geq n \geq 2$,
we focus on the Seifert 3-manifold $X$ of the type $ M(0; (1,0); (\ell,-1),(m,-1),(n,1))$; see, e.g., \cite{Kitano} for the definition.
Then, there is a genus-two Heegaard splitting of $X$, which yields the group presentation of $\pi_1(X)$ as $\langle g,h\ | \ r,s \rangle$ where $r=(gh)^nh^{-m}$ and $s=(hg)^ng^{-\ell}$.
Furthermore, as shown in \cite[p.127]{SieradskiIM86} and \cite[Section 7]{Wakijo},
the identity $W_X$ in Theorem \ref{hdihai2} is given by
\[ W_X=\rho_r h \rho_r^{-1} h^{-1} \rho_s g \rho_s^{-1} g^{-1}. \]
We now demonstrate all $ \SL_2(\F )$-representations of $X$, where we suppose that $\mathrm{char}(\F)=0$ and $\F$ contains the $ \mathrm{l.c.d.}(n,m,\ell)$-th root of unit. Then, it is known that all $ \SL_2(\F )$-representations of $X$ are already classified.
As a partial result, we mention the classification of such representations and the acyclicity as follows:
\begin{fact}[{\cite{Kitano}, \cite{Wakijo}}]\label{sl2-1}
Let $\alpha\in \F^{\times } \setminus \{\pm1\}$. 
Then, two matrices
$A=
\begin{pmatrix}
\alpha & 0 \\
0 & {\alpha}^{-1} \\
\end{pmatrix},
B=
\begin{pmatrix}
x & y \\
z & w \\
\end{pmatrix}
\in  \SL_2(\F)$
satisfy $(AB)^n=B^m,(BA)^n=A^\ell$ 
if and only if the following conditions hold.
\begin{enumerate}[(I)]
\setlength{\itemindent}{8mm}
\item $\alpha^\ell=\beta^m=\gamma^n\in \{ \pm 1\},$
\item $x=\dfrac{\alpha^{-1}(\beta+\beta^{-1})-(\gamma+\gamma^{-1})}{\alpha^{-1}-\alpha},$
$w=\dfrac{(\gamma+\gamma^{-1})-\alpha(\beta+\beta^{-1})}{\alpha^{-1}-\alpha}.$
\end{enumerate}
Here, $\beta,\beta^{-1}$ are the eigenvalues of $B$, and $\gamma,\gamma^{-1}$ are the eigenvalues of $AB$.

Furthermore, any irreducible representation $ \pi_1(X) \ra \SL_2(\F)$ is conjugate to the unitary representation $\rho_{A,B}: \pi_1(X) \ra \SU (2 )$ such that $\rho_{A,B}(g)=A$ and $\rho_{A,B}(h)=B$ for some matrices $A,B$ satisfying (I) and (II).
Furthermore, for such a representation $\rho_{A,B}$, the followings are known:
\begin{enumerate}[(i)] \setlength{\itemindent}{7mm}
\item If $\alpha^{\ell}=\beta^m=\gamma^n=-1$, then
the complex $C^{\rho_{A,B}}_*(X)$ is acyclic, and the torsion is given by
\[ \tau^0_{\rho_{A,B}}(M)=-\dfrac{4\alpha \beta \gamma}{(\alpha-1)^2(\beta-1)^2(\gamma-1)^2}\in \F^{\times }.\]
\item If $\alpha^{\ell}=\beta^m=\gamma^n=1$, then $C^{\rho_{A,B}}_*(X)$ is not acyclic.
\end{enumerate}
\end{fact}
Thus, as the remaining case for the computation of the torsion, we shall compute the refined torsion in case (ii).
\begin{prop}\label{sl2661}
Let $d$ be $\mathrm{l.c.d.}(n,m,\ell) \in \N$, and $\F$ be the cyclotomic field $\Q( \exp (2\pi \sqrt{-1}/  d))$.
Consider the unitary representation $\rho_{A,B}: \pi_1(X) \ra \SU(2)$ satisfying (ii) as above, and a hermitian $\rho_{A,B}$-invariant bilinear form $\psi $ defined by $\psi(v,w) := {}^t\overline{v}w$ for $v,w \in \F^2$.
Then, the cohomology is computed as
\[H^0_{\rho_{A,B}}(X;\F^2)=H^3_{\rho_{A,B}}(X;\F^2)=0, \qquad H^1_{\rho_{A,B}}(X;\F^2)=H^2_{\rho_{A,B}}(X;\F^2)=\F^2, \]
and the torsion in Theorem \ref{l33l24} is equal to $1$ in $\F/N(\F)$.
\end{prop}
Hereafter, for an integer $j $ with $0< j \leq n$, let $f^{(n)}_{j} \in \F^n$ be the vector in which the $j$-th entry is $1$ and the other entries are zeros; further, let $E_2$ denote the identity matrix of size two.
\begin{proof}
By the procedure in Section \ref{HG222}, the complex $C^{*}_{\rho_{A,B}} (X; \F^2)$ can be written as
\[ C^{*}_{\rho_{A,B}} (X; \F^2):0\lra \F^2 \xrightarrow{\small \begin{pmatrix}
E_2-A \\
E_2-B\\
\end{pmatrix} } \F^4 \overset{0}{\lra} \F^4 \xrightarrow{ (E_2-B,E_2-A)} \F^2\lra 0.\]
Since $E_2 -A $ and $E_2-B$ are invertible, we can immediately have $H^{0}_{\rho_{A,B}} (X ; \F^2) \cong H^{3}_{\rho_{A,B}} (X ; \F^2)\cong 0$.
Take a rational number
\[
r := \frac{1}{(1-\alpha)(1-\alpha^{-1})} + \frac{1}{(1-\beta)(1-\beta^{-1})} \in \Q^{\times},
\]
and two matrixes
\[
S := \begin{pmatrix}
(E_2-A^{-1})^{-1} \\
-(E_2-B^{-1})^{-1}
\end{pmatrix},\hspace{0.5cm}
T := \frac{1}{r} \begin{pmatrix}
-(E_2-B)^{-1} \\
(E_2-A)^{-1}
\end{pmatrix}.
\]
Then, we can verify that
the first cohomology $H^{1}_{\rho_{A,B}} (X ; \F^2) $ is generated by $S f^{(2)}_{1}, S f^{(2)}_{2}$, and
the second one is generated by $T f^{(2)}_{1}, T f^{(2)}_{2}$.
Furthermore, we can verify from \eqref{popo} that 
the hermitian form \eqref{kk4} with $(m,i)=(1,1)$ with respect to the generators $ S f^{(2)}_{1}, S f^{(2)}_{2} $ is represented by the diagonal matrix $(1)\oplus (1)$.
Thus, we can take the representative bases of $ B^* $ and $H^*$ as follows:
\begin{equation}\label{asdf1}
\widetilde{\mb{h}}_0 = \widetilde{\mb{h}}_0^{\rm dual} = \emptyset, \hspace{0.5cm}
\widetilde{\mb{h}}_1 = \{ S f^{(2)}_{1}, S f^{(2)}_{2} \} , \hspace{0.5cm}
\widetilde{\mb{h}}_1^{\rm dual} = \{ T f^{(2)}_{1}, T f^{(2)}_{2}\},
\end{equation}
\[
\widetilde{\mb{b}}_{1}= \{ f^{(2)}_{1}, f^{(2)}_{2} \} ,\hspace{0.5cm}
\widetilde{\mb{b}}_{2}= \widetilde{\mb{b}}_{4}= \emptyset,\hspace{0.5cm}
\widetilde{\mb{b}}_{3}= \{ f^{(4)}_{1}, f^{(4)}_{2}\}.\hspace{0.5cm}
\]
Hence, by the definition of the torsion  with an elementally intricate computation of linear algebra, we can compute the torsion as
\begin{equation*}
\begin{split}
\tau^0_{\rho_{A,B}} (X, \{ \mb{h}_{1}, \mb{h}_{1}^{\rm dual} \} ) &=r^2 \left( (1-\alpha)(1-\alpha^{-1}) + (1-\beta)(1-\beta^{-1})\right)^2 \\
&= 1 \in \F^{\times }/ \langle N(\F),\det(\pi_1(X))\rangle.
\end{split}
\end{equation*}
\end{proof}
Proposition \ref{sl2661} says that the torsion modulo $N(\F)$ is trivial; however, as in Proposition \ref{l32l245},
we focus on the case $b_1(X) \geq 1$ and compute the refinement of the torsion.
For this, we should detect the condition that $H_1 ( X; \Z)$ contains $\Z$:
\begin{lem}\label{pipu}
$H^1 (X;\Z) \cong \Z$ if and only if there exist $a,b,k \in \N$ such that $a$ is relatively prime to $b$ and $\ell =b(a+b)k, m=a(a+b)k$, and $n= a b k$.
\end{lem}
\begin{proof}
Considering the presentation of $\pi_1 ( X )= \langle g,h\ | \ r,s \rangle$, we can verify that $H^1 (X;\Z) \cong \Z$ if and only if $\ell,m,$ and $n$ satisfy an equation $m^{-1} + \ell^{-1} = n^{-1} $.
Moreover, by an elementary argument, we can see that the equation implies $\ell =b(a+b)k, m=a(a+b)k$ and $n= a b k$ for some $k \in \N$ and $a, b \in \Z$ as required. 
\end{proof}
To see the case $b_1(X)=1$,
by Lemma \ref{pipu}, it suffices to consider the case where $a,b \in \N$ are relatively prime and $\ell =b(a+b)k, m=a(a+b)k$, and $n= a b k$ for some $k\in\N$.
Then, the epimorphism $\alpha_{a,b} : \pi_1 (X) \ra \Z$ defined by setting $\alpha_{a,b} (g) = a$ and $\alpha_{a,b} (h) = b$ is a generator of
$H^1 (X;\Z) \cong \Z$.
In this situation, we next confirm the admissibility in Proposition \ref{l32l2451}.
\begin{lem}\label{pipu2}
Suppose $H^1 (X;\Z) \cong \Z$.
Consider the unitary representation $\rho_{A,B}: \pi_1(X) \ra \SU(2)$ satisfying (ii) above, and the hermitian $\rho_{A,B}$-invariant bilinear form $\psi(v,w) := {}^t\overline{v}w$ for $v,w \in \F^2$.
Then, the cohomology is computed as
\[ H^0_{\rho_{A,B}} \oplus H^3_{\rho_{A,B}}  (X;\F^2)=0, \qquad \textrm{and} \qquad H^1_{\rho_{A,B}}(X;\F^2)=H^2_{\rho_{A,B}}(X;\F^2)=\F^2, \]
and the triple $(\rho_{A,B}, \psi, \Sigma_{\alpha_{a,b}})$ is admissible.
\end{lem}
\begin{proof}
Under the same notation as in Theorem \ref{sl2661},
$\Psi$ denotes the matrix of the anti-hermitian form \eqref{p6p6} with $(\rho_{A,B}, \psi, \Sigma_{\alpha_{a,b}}) $ related to the basis $\widetilde{\mb{h}}_1$.
Then, with the help of a computer program, we can verify that $\Psi$ is equal to an anti-hermitian matrix
\begin{equation*}
\begin{split}
&\dfrac{1}{ab(a+b)k} (E_2-A)^{-1} \left( (E_2-A)^{-1} - (E_2-AB)^{-1} \right) (E_2-A^{-1})^{-1} \\
& \quad +(E_2-A)^{-1} (B-A^{-1})^{-1} (E_2-B^{-1})^{-1} \\
& \quad +(E_2-B)^{-1} (B^{-1}-A)^{-1} (E_2-A^{-1})^{-1}\\
& \quad +(E_2-B)^{-1} \left( (E_2-B)^{-1} - (E_2-BA)^{-1} \right) (E_2-B^{-1})^{-1},
\end{split}
\end{equation*}
and that $\Psi$ has eigenvalues $\sqrt{-1} u^2$ and $-\sqrt{-1} u^2$.
Here, $u^2= a b(a+b)k r^2 /( \lvert 1-\alpha \rvert\,\lvert 1-\beta \rvert\,\lvert 1-\gamma \rvert) \in \F$.
Thus, $(\rho_{A,B}, \psi, \Sigma_{\alpha_{a,b}})$ is admissible as desired.
\end{proof}
Thus, by Lemma \ref{pipu2} and using $u$ above, we can choose a unitary matrix $R_{a,b} \in\mathrm{SU}(2)$ such that
\[
\overline{(uR_{a,b})}^{\mathrm{transpose}} M (uR_{a,b})
=
\sqrt{-1}\begin{pmatrix}-1 & 0 \\0 & 1 \end{pmatrix} .
\]
Here, $M$ is the matrix presentation of the pairing in \eqref{p6p6}. 

\begin{exa}\label{sl2688865}
By the help of a computer program, we can explicitly describe the matrix $R_{a,b}$; however,
the explicit description is generally complicated.
For instance, even if $\alpha \beta = 1$ and
\[
B=
\left(
\begin{array}{cc}
\cos\theta & -\sin\theta \\
\sin\theta & \cos\theta \\
\end{array}
\right)
\left(
\begin{array}{cc}
\alpha^{-1} & 0 \\
0 & \alpha \\
\end{array}
\right)
\left(
\begin{array}{cc}
\cos\theta & \sin\theta \\
-\sin\theta & \cos\theta \\
\end{array}
\right)
\] for some $0 < \theta < \frac{\pi}{2}$, then
$R_{a,b}=
\frac{1}{\sqrt{2}}
\left(
\begin{array}{cc}
-\sqrt{1-\sin\theta} & -\sqrt{1+\sin\theta} \\
\sqrt{1+\sin\theta} & -\sqrt{1-\sin\theta} \\
\end{array}
\right),$ and hence $\mathrm{det}R_{a,b}=1$.
\end{exa}

\begin{prop}\label{sl2665}
Under the same notation in Lemma \ref{pipu2}, the torsion in Theorem \ref{l32l245} is equal to
\[a^2 b^2 (a+b)^2 k^2 \dfrac{(\lvert 1-\alpha \rvert^2 + \lvert1-\beta\rvert^2)^4}{\lvert1-\alpha\rvert^6 \, \lvert1-\beta\rvert^6 \, \lvert1-\gamma\rvert^2} \mathrm{det}R_{a,b} \in \C^{\times }/\{ \pm 1\} . \]
\end{prop}
\begin{proof}
Let us replace $\widetilde{\mb{h}}_0, \widetilde{\mb{h}}_1, \widetilde{\mb{h}}_2$ and $\widetilde{\mb{h}}_3$ in \eqref{asdf1} with
\[\begin{split}
&\widetilde{^{\backprime}\mb{h}}_0 = \widetilde{^{\backprime}\mb{h}}_0^{\rm dual} = \emptyset, \ \ \ \
\widetilde{^{\backprime}\mb{h}}_1 = (uSR_{a,b} f^{(2)}_{1} , uSR_{a,b} f^{(2)}_{2}), \\
&\widetilde{^{\backprime}\mb{h}}_1^{\rm dual} = \left(\frac{r}{u}TR_{a,b} f^{(2)}_{1}, \frac{r}{u}TR_{a,b} f^{(2)}_{2} \right),
\end{split}
\]
respectively. Then, we can check that $ \widetilde{^{\backprime}\mb{h}}_1$ is an orthogonal basis with respect to the pairing \eqref{p6p6}.
Hence, from the definition of the torsion, we get the following computation:
\begin{equation*}
\begin{split}
\tau^0_{\rho_{A,B}} (X, \{ ^{\backprime}\mb{h}_{1}, ^{\backprime}\mb{h}_{1}^{\rm dual} \}) &=\dfrac{u^4}{r^2} (\lvert 1-\alpha \rvert^2\, + \lvert 1-\beta \rvert^2\,)^2 \mathrm{det}R_{a,b} \\
&= a^2 b^2 (a+b)^2 k^2 \dfrac{(\lvert 1-\alpha \rvert^2 + \lvert1-\beta\rvert^2)^4}{\lvert1-\alpha\rvert^6 \, \lvert1-\beta\rvert^6 \, \lvert1-\gamma\rvert^2} \mathrm{det}R_{a,b} \in \F^{\times }.
\end{split}
\end{equation*}
\end{proof}

\subsection{Computation 2; Torus bundles}\label{HG447}
In this subsection, when $G = \SU(2)$ and $\F=\C$, we will compute the adjoint torsions of torus bundles over $S^1$.

First, we review torus bundles over $S^1$.
Let $T^2$ be the torus. We fix integers $\alpha ,\beta ,\gamma , \delta$ such that $\alpha \delta -\beta \gamma =\pm 1$. Then, we have a homeomorphism
$f : T^2 \ra T^2 $ such that the induced map $f_*: \pi_1(T^2 ) \ra \pi_1(T^2)$ is equal to $ \begin{pmatrix}
\alpha & \gamma \\
\beta & \delta
\end{pmatrix} \in \mathrm{GL}_2(\Z)$.
The mapping torus $T_{f}$ is the quotient space of $T^2 \times [0,1]$ subject to the relation $(y,0) \sim (f(y),1)$ for any $y \in T^2$.
Any torus bundle is known to be homeomorphic to $T_f$ for some $f$.

We will construct an identity of the torus bundle. Choose a generating set $\{ a,b \} $ of $\pi_1(T^2 )$, which gives the isomorphism $\pi_1(T^2 ) \cong \langle
a,b \mid aba^{-1}b^{-1} \rangle . $ Following a van Kampen argument, we can verify the presentation of $ \pi_1(T_{f})$ as
\begin{equation}\label{m3as2} \langle \ a,b,t \mid t a^{\alpha} b^{\beta}t^{-1}a^{-1}, \ t a^{\gamma} b^{\delta}t^{-1}b^{-1}, aba^{-1}b^{-1} \ \rangle.
\end{equation}
As a special case, suppose some $q_f$ in the free group $ \langle a ,b \mid \rangle $ satisfying
\begin{equation}\label{m3as442} f_* (aba^{-1}b^{-1})=q_f (aba^{-1}b^{-1}) q_f^{-1} \in \langle a,b \mid \rangle ,
\end{equation}
and consider the identity $W$ of the form
\[ \rho_1 (a \rho_2 a^{-1}) (aba^{-1})\rho_1^{-1} (aba^{-1})^{-1} (aba^{-1}b^{-1}) \rho_2^{-1}(aba^{-1}b^{-1} )^{-1} \rho_3 t q_f \rho_3^{-1} q_f^{-1} t^{-1}. \]
Then, 
$W$ is known to be an identity satisfying the condition in Theorem \ref{hdihai2}; see, e.g., \cite[Theorem 3.1]{Nos2}.

As is known \cite{Kitano}, the $\SU(2)$-character variety $R_{\SU(2)}(T_f)$ as an open dense subset of $ \Hom(\pi_1(T_f), \SU(2))$/conj is
homeomorphic to
the set
\begin{equation}\label{oioioi} \{ (u,v) \in \C^2 \mid u^{\alpha +1} v^{\beta} =1 , \ u^{\gamma} v^{\delta+1} =1, \ (u^2 , v^2) \neq (1,1), \ |u|=|v|=1 \} .
\end{equation}
We may suppose the open set $R_{\SU(2)}(T_f)^o $ to be the set \eqref{oioioi}.
In particular, $ R_{\SU(2)}(T_f)$ is not of finite order
if and only if $\alpha = \delta = -1$ and $\beta \gamma =0$, and if and only if the cohomology $H^1_{\mathrm{Ad}\rho }(T_f ;\mathfrak{su}(2))$  is not zero. Here, $\mathrm{Ad}\rho$ is the adjoint action of a representation $\rho : \pi_1(T_f)\ra \SU(2)$.
Since we have interests in acyclic cases, we may focus on such a case and suppose
$\alpha = \delta = -1$ and $ \gamma =0$, as in \eqref{m3as442}.
For $(u,v)$ in the set \eqref{oioioi}, we define
a representation $\rho_{u,v}: \pi_1(T_f) \ra \SU(2)$ by
\[ \rho_{u,v} (a) = \begin{pmatrix}
u & 0 \\
0& u^{-1} \\
\end{pmatrix} , \ \ \ \rho_{u,v} (b) = \begin{pmatrix}
v& 0 \\
0& v^{-1} \\
\end{pmatrix} , \ \ \ \rho_{u,v} (t) = \begin{pmatrix}
0& 1 \\
-1 & 0 \\
\end{pmatrix} .\]
Then, the correspondence $(u,v) \mapsto \rho_{u,v}$ gives rise to an injection to an open dense subset $R_{\SU(2)}(T_f)$. 

Under a canonical identification $\mathfrak{su}(2)\otimes \C \cong \mathbb{C}^3$,
the adjoint action is obtained by
\[ \begin{split}\mathrm{Ad}\rho_{u,v} (a) &= \begin{pmatrix}
u^2 & 0 &0 \\
0 & 1&0 \\
0& 0 & u^{-2} \\
\end{pmatrix} , \ \ \ \  \mathrm{Ad}\rho_{u,v} (b) = \begin{pmatrix}
v^2 & 0 &0 \\
0 & 1&0 \\
0& 0 & v^{-2} \\
\end{pmatrix} , \\  \mathrm{Ad}\rho_{u,v}(t) &= \begin{pmatrix}
0 & 0 &-1 \\
0 & -1&0 \\
-1& 0 & 0 \\
\end{pmatrix} .\end{split}\]
For each $\rho_{a,b}$, we now compute the torsion of $T_f$ in Theorem \ref{l33l24}.

\begin{prop}\label{sl2675} Let $G=\SU(2)$. Suppose $\alpha = \delta =-1, \gamma =0$, and $ \beta \neq 0.$
Let $\psi$ be the Killing form. 
Then, for the representation $\rho_{u,v} $, the cohomology is computed as
\[ H^0_{ \mathrm{Ad}\rho_{u,v}} \oplus H^3_{ \mathrm{Ad}\rho_{u,v}} (T_f;\C^3)=0, \ \  \textrm{and} \ \ H^1_{ \mathrm{Ad}\rho_{u,v}} (T_f;\C^3) \cong H^2_{ \mathrm{Ad}\rho_{u,v}} (T_f;\C^3) \cong \C , \]
and the torsion $\tau^0_{ \mathrm{Ad}\rho_{u,v}} (X, \{ \mb{h}_{1}, \mb{h}_{1}^{\rm dual} \} )\mb{h}_{1}^{\otimes 2}$ in Theorem \ref{l33l24} is $-4/ \beta \mb{h}_{1}^{\otimes 2}$.
In addition, the volume in Definition \ref{l8435} is $4 \beta^{1/2} \pi $.
\end{prop}
\begin{proof}
Based on the procedure defined in Section \ref{HG222}, $C^*_{\mathrm{Ad}\rho_{u,v}}(X ; \C^3)$ is described as
$$C^*_{\mathrm{Ad}\rho_{u,v}}(X ; \C^3):0\lra \C^3 \overset{\delta^0}{\lra} \C^9 \overset{\delta^1}{\lra} \C^9 \overset{\delta^2}{\lra} \C^3\lra 0,$$
where $\delta^0 , \delta^1$, and $\delta^2$ are respectively defined by
\begin{equation*}
\begin{split}
\delta^0 (x)&=
\begin{pmatrix}
1-u^2 & 0 & 0 \\
0 & 0 & 0 \\
0 & 0 & 1-u^{-2} \\
1-v^2 & 0 & 0 \\
0 & 0 & 0 \\
0 & 0 & 1-v^{-2} \\
1 & 0 & 1 \\
0 & 2 & 0 \\
1 & 0 & 1
\end{pmatrix}x,
\\
\delta^1(y)&=
\begin{pmatrix}
-1 & 0 & u^2 & 0 & 0 & 0 & 1-u^2 & 0 & 0 \\
0 & 0 & 0 & 0 & -\beta & 0 & 0 & 0 & 0 \\
u^{-2} & 0 & -1 & 0 & 0 & 0 & 0 & 0 & 1-u^{-2} \\
0 & 0 & 0 & -1 & 0 & v^2 & 1-v^2 & 0 & 0 \\
0 & 0 & 0 & 0 & 0 & 0 & 0 & 0 & 0 \\
0 & 0 & 0 & v^{-2} & 0 & -1 & 0 & 0 & 1-v^{-2} \\
1-v^2 & 0 & 0 & u^2-1 & 0 & 0 & 0 & 0 & 0 \\
0 & 0 & 0 & 0 & 0 & 0 & 0 & 0 & 0 \\
0 & 0 & 1-v^{-2} & 0 & 0 & u^{-2}-1 & 0 & 0 & 0
\end{pmatrix}y,
\\
\delta^2(y)&=
\begin{pmatrix}
1-v^2 & 0 & 0 & u^2-1 & 0 & 0 & 1 & 0 & u^2 v^2 \\
0 & 0 & 0 & 0 & 0 & 0 & 0 & 2 & 0 \\
0 & 0 & 1-v^{-2} & 0 & 0 & u^{-2}-1 & u^{-2} v^{-2} & 0 & 1
\end{pmatrix}y.
\end{split}
\end{equation*}
for $x \in \C^3$ and $y\in\C^9$.
Accordingly, we can take bases of $H^*$ and $B^*$ of the forms
\begin{equation}\label{i9i9}
\widetilde{\mb{h}}_0 = \widetilde{\mb{h}}_0^{\rm dual} = \emptyset, \hspace{0.5cm}
\widetilde{\mb{h}}_1 = \{ f^{(9)}_2\} , \hspace{0.5cm}
\widetilde{\mb{h}}_2 = \{ f^{(9)}_2\},
\end{equation}
\[
\tilde{\textbf{b}}_1=
\{ f^{(3)}_{1}, f^{(3)}_{2},f^{(3)}_{3}\}
,\ \ 
\tilde{\textbf{b}}_2=\{ f^{(9)}_{1}, f^{(9)}_{3},f^{(9)}_{5},f^{(9)}_{7},f^{(9)}_{9}\}
,\ \ 
\tilde{\textbf{b}}_3=
\{ f^{(9)}_{1}, f^{(9)}_{3},f^{(9)}_{8}\}.
\]
We can confirm that the matrix of the hermitian form \eqref{kk4} with respect to the bases $\widetilde{\mb{h}}_1$ and $\widetilde{\mb{h}}_2$ coincides with $1$. 
Furthermore, since $q_f$ in \eqref{m3as442} equals $ a^{-1}b^{-1}$, we can compute the pairing in \eqref{p6p6}; particularly, we can confirm
$\widetilde{\mb{h}}_1^{\text{dual}} = \widetilde{\mb{h}}_2$.
Hence, we can compute the torsion $ \tau^0_{\mathrm{Ad}\rho_{u,v}} (T_f, \{ \mb{h}_{1}, \mb{h}_{1}^{\rm dual} \} )$ as $-4/\beta $ from the definition of torsions.

Finally, to compute the volume, we note from \eqref{oioioi}, that the set $ R_{\SU(2)}(T_f) $ is homeomorphic to 
$\sqcup^{\beta} S^1$ almost everywhere. Furthermore,
by a dual argument,
the torsion $ \tau^0_{ \mathrm{Ad}\rho_{u,v}} (T_f, \{ \mb{h}_{2}, \mb{h}_{2}^{\rm dual} \} )^{1/2} $ as $ 2 / \beta^{1/2} \mb{h}_{2}^{\otimes} \in H^2( T_f ; \mathfrak{su}(2)) $ is a constant form on $ R_{\SU(2)}(T_f) $. 
Thus, the volume is $ (2 / \beta^{1/2} ) \times 2 \pi \beta= 4 \beta^{1/2} \pi . $
\end{proof}
However, the torsion $\tau^0_{\mathrm{Ad}\rho_{u,v}} (T_f, \{ \mb{h}_{1}, \mb{h}_{1}^{\rm dual} \} ) = - \beta /4 \in \R$ modulo $\mathbb{C}^{\times}/N(\mathbb{C} ) \cong S^1$
is $-1$ trivially.
Thus, it is worth considering the refinements in Section \ref{HG99}.
\begin{prop}\label{lli}
Under the same assumptions of Proposition \ref{sl2675}, for any $\mathrm{Ad}\rho_{u,v}$-invariant bilinear form $\psi$ and surjection $\alpha : H_1 (T_f ) \ra \Z$, the pair $(\mathrm{Ad}\rho_{u,v}, \psi,\Sigma_\alpha)$ is not admissible.
\end{prop}

\begin{proof}
There are only two surjections $\alpha_{+}, \alpha_{-} : H_1 (T_f ) \ra \Z$ that satisfy $\alpha_{+} (t) = -\alpha_{-} (t)= 1$ and $\alpha_{\pm} (a) = \alpha_{\pm} (b) = 0$.
It follows easily that $-\Sigma_{\alpha_+} = \Sigma_{\alpha_-} = (1 \otimes f^{(3)}_3)\in \Z \otimes C_1(\widetilde{X} ;\Z)$.
In addition, $(f^{(9)}_2 \otimes f^{(9)}_2) (\Upsilon(a b a^{-1} b^{-1})) = 0 \in \C^3 \otimes \C^3.$
Hence, $\psi_* ( (f^{(9)}_2 \smile f^{(9)}_2) \cap \Sigma_{\alpha_\pm})=0$ for any $\mathrm{Ad}\rho_{u,v}$-invariant bilinear form $\psi$, which implies that the pairs $(\mathrm{Ad}\rho_{u,v}, \psi,\Sigma_{\alpha_{\pm}})$ are not admissible.
\end{proof}
Thus, we can not apply the torsion to Proposition \ref{l32l245}.
On the other hand, the refinement in Proposition \ref{l32l2451} is applicable, and shown to be non-trivial as follows:
\begin{prop}\label{llti5}
Let $\Sigma$ be a fiber of $T_f$ homomorphic to $T^2$.
Then, the inclusion $\Sigma \hookrightarrow T_f $ induces the isomorphism $ H^2_{ \mathrm{Ad}\rho_{u,v}}(T_f ;\C^n ) \cong H^2_{ \mathrm{Ad}\rho_{u,v}}( \Sigma ;\C^n ) \cong \C$ as in Proposition \ref{l32l2451}.
Thus, the refined torsion \eqref{kk5344} is equal to $| \beta |/4$ in $\mathbb{C}^{\times}/\{ \pm 1\}. $
\end{prop}
The proof is obvious by the description of $ \widetilde{\mb{h}}_2 $ in \eqref{i9i9}.

\subsection*{Acknowledgments}
The authors sincerely express their gratitude to Yoshikazu Yamaguchi for the valuable comments.
The authors are also grateful to an anonymous referees for their careful reading of the paper.

\appendix
\section{Proof of Proposition \ref{l33334}}\label{9991}

Hereafter, we denote the local coefficient module $\F^n$ by $A$ and
the dual of the conjugate representation $\overline{\mathrm{Hom}(A; \F)}$ by $A^*$. We
have an isomorphism $\mathcal{D}: A \cong A^*$ of local coefficients by the correspondence $a \mapsto \psi ( \bullet, a )$.
We notice that $(A^*)^* \cong A$, and $\mathcal{D}$ gives rise to an isomorphism
$ \mathcal{D}_*:H^*_{\rho}(X ;A) \cong H^*_{\rho}(X ;A^*). $
\begin{proof}[Proof of Proposition \ref{l33334}] 
For simplicity, we abbreviate $\rho$ hereafter.
Since Char$\F = 0$, we have $\mathrm{Ext}^1(H^p(X; A); \F) = 0$; therefore, by the
universal coefficient theorem, the Kronecker product gives an isomorphism
\[ \mathcal{K}: H^p (X; A^*) \stackrel{\sim}{ \lra } \mathrm{Hom}( H_p (X; (A^*)^*) ;\F )\cong \mathrm{Hom}( H_p (X; A),\F), \]
which implies $\mathrm{dim} H^n (X; A^*) = \mathrm{dim} H_n (X; A) $. Thus, Poincar\'{e} duality leads to
\[\dim H^p(M;A) = \dim H_{2m+1-p}(M;A) = \dim H^{2m+1-p}(M;A) .\]
Next, to show the non-degeneracy of the pairing \eqref{kk4}, we examine the composite of $\psi$ and the cap-product
\begin{equation}\label{popo1} \cap_{\psi}:= \psi \circ \cap : H_p(X;A) \otimes H^p(X;A) \lra \F. \end{equation}
Furthermore, consider the following coupling:
\begin{equation}\label{popo2} H_p(X;A) \otimes H^p(X;A) \lra \F ; \ \ \ a \otimes f \longmapsto \mathcal{K}(\mathcal{D}_*(f))(a) . \end{equation}
By the definitions of the cap-product, $ \mathcal{D}_*$, and $\mathcal{K}$, the maps \eqref{popo1} and \eqref{popo2} are equal up to signs. Thus,
the Kronecker delta is non-degenerate, and so is the map $\cap_{\psi}$.

Finally, we will complete the proof.
Suppose $a \in H^p(X;A) $ such that $a \smile_{\psi} b=0 $ for any $b \in H^{2m+1-p} (X;A)$. It suffices to show that $a = 0$. Note that
\[ a \smile_{\psi} b= \psi( [X] \cap (a \smile b))= \psi(( [X] \cap a ) \cap b) = \cap_{\psi}( [X] \cap a , b) . \]
Since $\cap_{\psi}$ is non-degenerate as above, $[X] \cap a =0$. Hence, Poincar\'{e} duality deduces $a = 0$ as desired.
\end{proof}

\section{Comparison with the works of Farber and Turaev}\label{999166}
As mentioned in Section \ref{HG}, we will show that our refined torsion in Theorem \ref{l33l24} is
a modification of the Poincar\'{e}-Reidemeister (scalar) product \cite{Fa,FT}.

Throughout this appendix, as in Section \ref{HG}, we suppose a closed manifold $X$ of dimension $2m+1$ with an $\SL_n$-representation $\rho: \pi_1(X) \ra \SL_n(\F)$ and
a (anti-)hermitian $\rho$-invariant bilinear form $\psi: \F^n \times \F^n \ra \F $.
Here $\F$ is an arbitrary field of characteristic zero\footnote{In the original papers \cite{Fa, FT}, the authors considered only a flat vector bundle with unimodularity and a Poincar\'{e} duality, where $\F$ is either $\R$ or $\mathbb{C}$.
However, the duality in \cite{Fa} contained an ambiguous concept; thus, to clarify the duality as in Proposition \ref{l33334}, this paper supposes such a $\psi$ and deals with fields of characteristic zero. }.
The torsion $\tau^0_{\rho}$ defined in this paper is valued in $\F^{\times }/N(\F)$; see \eqref{kk54}.
Here, let us note the following lemma:
\begin{lem}\label{exa15}
Let $U$ be a 1-dimensional $\F$-space with a basis $u_0\in U .$
Then, the set of the isomorphism classes of non-trivial hermitian bilinear forms $\phi: U\times U \ra \F$ is in a 1:1-correspondence
with $ \F^{\times} / N(\F)$. Here, the correspondence is given by $\phi \mapsto [\phi(u_0,u_0) ]$ as the discriminant, and
it is independent of the choice of the basis $u_0$.
\end{lem}
\noindent
Thus, it is sensible to study the torsion as a hermitian bilinear form as follows.

For this, 
let us establish some terminology.
For a finite dimensional vector space $W$, let $\det (W)$ denote the exterior product $\Lambda^{\dim W} (W) $, and
$\det (W)^{-1}$ denote the dual $\Hom_{\F} ( \det (W) ,\F)$. Here, if $W= \{0\}$, then $ \det (W)= \det (W)^{-1}=\F$ by definition.
For a finite-dimensional graded vector space $V= V_0\oplus V_1 \oplus \cdots \oplus V_k$,
{\it the determinant line of $V$} is defined to be a $1$-dimensional vector space of the form
\[ \det V= \det (V_0)\otimes \det (V_1)^{-1} \otimes \det (V_2)\otimes\cdots \otimes \det (V_k)^{(-1)^k}. \]

For a basis $ \mathbf{h}= (\mathbf{h}_0 ,\dots, \mathbf{h}_{2m+1} ) $ of $H^*(X ; \F^n) $, we write $ \mathbf{h}^{\otimes } $ for
\begin{equation*}
\mathbf{h}_0 \otimes \mathbf{h}_1^{-1} \otimes \mathbf{h}_2 \otimes \cdots \otimes \mathbf{h}_{2m+1}^{-1} \in \det (H^*(X ; \F^n) ).
\end{equation*}
Then, using the refined torsion $ \tau_{\rho}^0$ in \eqref{pp4577}, we define $ \tau_{\rho}(X ) \in \det H^*(X;\F^n) $ to be $ \tau_{\rho}^0(X , \mathbf{h} ) \mathbf{h}^{\otimes}$. Here, $ \tau_{\rho}(X ) $ is known to be independent of the choices of $ \mathbf{h}$; see \cite{Fa,FT}.

We set the conjugate representation $\bar{\rho}: \pi_1( X) \ra \SL_n(\F)$ defined by $\bar{\rho}(g)=\overline{\rho(g)}$ for $g \in \pi_1(X)$,
and the associated cohomology by $H^*(X; (\F^n)^* ) $. Furthermore, we consider the direct sum $\rho \oplus \bar{\rho}: \pi_1( X) \ra \SL_{2n} (\F)$.
Recall from Proposition \ref{l33334} the isomorphism $D$ as a Poincar\'{e} duality,
which induces an isomorphism $D: \det H^*(X;\F^n) \ra \det H^*(X;(\F^n)^* ) $.
In addition, define
\[
\begin{aligned}
\mu^{\bullet} : \det H^*(X;\F^n) \otimes \det H^*(X;(\F^n)^* ) &\lra \det H^*(X;\F^n \oplus (\F^n)^* ); \\
(v_1 \wedge \cdots \wedge v_m, w_1 \wedge \cdots \wedge w_m) &\longmapsto v_1 \wedge \cdots \wedge v_m \wedge w_1 \wedge \cdots \wedge w_m .
\end{aligned}
\]
As in Section 4.7 in \cite{Fa} and Section 9 in \cite{FT}, the {\it Poincar\'{e}-Reidemeister (scalar) product} is defined to be a hermitian product by setting 
\begin{equation}\label{jlijli} {\rm PR}: \det H^*(X ;\F^n) \times \det H^*(X ;\F^n ) \ra \F^{\times } ; ( \alpha , \beta) \mapsto \frac{\mu^\bullet (\alpha \otimes D(\beta) )}{ \tau_{\rho \oplus \bar{\rho}}(X) } .\end{equation}

However, we will show that this product is trivial in the sense of Lemma \ref{exa15}.
For this, 
let $ H^{\rm half}(X;\F^n) $ be the graded vector space $ \oplus_{j=0}^m H^{j}(X;\F^n) $.
Similarly, the Poincar\'{e} duality induces an isomorphism
\begin{equation*}
D : \det H^{\rm half} \ra \det(\oplus_{j=m+1}^{2m+1} H^{ j}(X;\F^n) ).
\end{equation*}
Then, by definition, for any $\mathbf{h}^{\rm half} \in H^{\rm half}(X;\F^n) $, we can easily show
\[{\rm PR}(\tau_{\rho}^0(X , \mathbf{h}') \mathbf{h}', \tau_{\rho}^0(X , \mathbf{h}') \mathbf{h}') =1 \in \F, \]
where $\mathbf{h}'= \mathbf{h}^{\rm half} \otimes D (\mathbf{h}^{\rm half})^{-1}$.
Hence, by Lemma \ref{exa15}, the bilinear isomorphism class of ${\rm PR} $ is trivial, as required.

We shall consider a modification of the PR product, and provide a comparison with the refined torsion. To be precise, we define {\it a half PR product} by the hermitian product
\begin{equation}\label{jlijli2} {\rm PR}^{\rm half}: \det H^{\rm half}(X ;\F^n) \times \det H^{\rm half}(X ;\F^n ) \ra \F^{\times } ; ( \alpha , \beta) \mapsto \frac{\mu^\bullet (\alpha \otimes D(\beta) )}{ \tau_{\rho }(X) } .\end{equation}
By definition, we can easily check that, for any $\mathbf{h}^{\rm half} \in H^{\rm half}(X;\F^n) $, the value $ {\rm PR}^{\rm half}(\mathbf{h}^{\rm half}, \mathbf{h}^{\rm half} )$
is equal to the inverse of the refined torsion in \eqref{kk54}.
In conclusion, it follows from Lemma \ref{exa15} that the bilinear form associated with the refined torsion in \eqref{kk54} is equal to the half PR product \eqref{jlijli2}, as a modification of the original ${\rm PR} $ product.

%


\begin{thebibliography}{99}



\bibitem[Dub]{Dub} J. Dubois, {\it Torsion de Reidemeister non ab\'{e}lienne et forme volume sur l’espace
des repr\'{e}sentations du groupe d’un nud}, Ph.D. thesis, Universit\'{e} Blaise Pascal,
http://tel.ccsd.cnrs.fr/documents/archives0/00/00/37/82, 2003.


\bibitem[Far]{Fa}
M.~Farber, \emph{Combinatorial invariants computing the {R}ay-{S}inger analytic
torsion}, Differential Geom. Appl. \textbf{6} (1996), no.~4, 351--366.

\bibitem[FT]{FT}
M.~Farber and V.~Turaev, \emph{Poincar\'{e}-{R}eidemeister metric, {E}uler
structures, and torsion}, J. Reine Angew. Math. \textbf{520} (2000),
195--225.




\bibitem[HSW]{HSW} Hillman, J.A.; Silver, D. S.; Williams, S.G. {\it On reciprocality of twisted Alexander invariants}. Algebr. Geom. Topol. \textbf {10} (2010), no. 2, 1017--1026.

\bibitem[Kit]{Kitano}
T. Kitano, \emph{Reidemeister torsion of {S}eifert fibered spaces for {${\rm
SL}(2;{\bf C})$}-representations}, Tokyo J. Math. \textbf{17} (1994), no.~1,
59--75.




\bibitem[Mil]{Mil} J. Milnor, {\it Whitehead torsion}, Bull. Amer. Math. Soc. {\bf 72} (1966), 358--426.

\bibitem[MP]{MP}
M. Mulase and M. Penkava, {\it Volume of representation varieties,} arXiv:math/0212012, 2002.



\bibitem[NS]{NS}
F.~Naef and P.~Safronov, \emph{Torsion volume forms}, arXiv preprint arXiv:2308.08369, 2023.

\bibitem[Nic]{Nic}
L. I. Nicolaescu, {\it The Reidemeister torsion of 3-manifolds}, volume 30 of de Gruyter Studies
in Mathematics, Walter de Gruyter \& Co., Berlin (2003).

\bibitem[Nos]{Nos2}T.~Nosaka, {\it Cellular chain complexes of universal covers of some 3-manifolds}, J. Math. Sci. Univ. Tokyo \textbf{29} (2022), no. 1, 89--113.



\bibitem[Kar]{lll} Y. Karshon, {\it An algebraic proof for the symplectic structure of moduli space}, Proc.
Amer. Math. Soc. 116 (1992), no. 3, 591--605. MR 1112494

\bibitem[PY]{PY}
J. Porti and S. Yoon. {\it The adjoint Reidemeister torsion for the connected sum of knots}, Quantum Topol. \textbf{14} (2023), no. 3, 407--428.

\bibitem[Sic]{Wei}
A. S. Sikora, {\it Character varieties}, Trans. Amer. Math. Soc. \textbf{364} (2012), 5173--5208.

\bibitem[S80]{SieradskiMZ80} A. J. Sieradski, {\it Framed links for Peiffer identities}, Math. Z. {\bf 175} (1980), no. 2, 125--137.

\bibitem[S86]{SieradskiIM86} A. J. Sieradski, {\it Combinatorial squashings, 3-manifolds, and the third homology of groups}, Invent. math. {\bf 84} (1986), no. 1, 121--139.

\bibitem[S\"oz]{Sozen}
Y.~S\"{o}zen, \emph{Symplectic chain complex and {R}eidemeister torsion of
compact manifolds}, Math. Scand. \textbf{111} (2012), no.~1, 65--91.


\bibitem[Tro]{TrotterAM62} H. F. Trotter, {\it Homology of group systems with applications to knot theory}, Ann. Math. {\bf 76} (1962), no.~2, 464--498.

\bibitem[Tur]{Tur}V. Turaev, {\it Introduction to combinatorial torsions}, Lectures in Mathematics, ETH
Z\"{u}rich, Birkh\"{a}user, Basel (2001).

\bibitem[Wak]{Wakijo}
N.~Wakijo, \emph{Twisted {R}eidemeister torsions via {H}eegaard splittings},
Topol. Appl. \textbf{299} (2021), 107731, 22.





\bibitem[Wit]{Wit}
E.~Witten, \emph{On quantum gauge theories in two dimensions}, Comm. Math.
Phys. \textbf{141} (1991), no.~1, 153--209.

\bibitem[Yam]{Yam}Y. Yamaguchi, {\it A relationship between the non-acyclic Reidemeister torsion and a zero of
the acyclic Reidemeister torsion}, Ann. Inst. Fourier (Grenoble) 58 (2008), no. 1, 337--362.




%
%
\end{thebibliography}
\end{document}